\documentclass[a4paper,11pt]{amsart}
\usepackage[utf8]{inputenc}
\usepackage{amsmath, amssymb}
\usepackage{setspace}
\usepackage[left=2cm, right=2cm, top=2cm, bottom=2cm]{geometry}                 

\usepackage{graphicx} 
\usepackage{pgf,tikz}
\usetikzlibrary{arrows}
\usepackage{amssymb}
\usepackage{pdfsync}
\usepackage{mathrsfs}
\usepackage{hyperref} 
\usepackage{verbatim} 
\usepackage{epstopdf}
\usepackage{xargs}
\usepackage{todonotes}
\usepackage[normalem]{ulem}

\usepackage{mathtools}

\mathtoolsset{showonlyrefs}

\newcommandx{\jow}[2][1=]{\todo[linecolor=orange,backgroundcolor=orange!25,bordercolor=orange,#1]{#2}}

\newcommandx{\mateus}[2][1=]{\todo[linecolor=blue,backgroundcolor=blue!25,bordercolor=blue,#1]{#2}}

\def\R{\mathbb{R}}

\def\N{\mathbb{N}}

\newcommand{\eps}{\varepsilon}

\renewcommand{\H}{\mathcal{H}}
\newcommand{\AC}{\mathcal{A}}
\newcommand{\BC}{\mathcal{B}}
\newcommand{\CC}{\mathcal{C}}

\newcommandx{\eq}{\approxeq}
\newcommand{\supp}{\mathrm{supp}}
\renewcommand{\a}{\mathbf{a}}
\renewcommand{\b}{\mathbf{b}}

\newtheorem{theorem}{Theorem}
\newtheorem{corollary}[theorem]{Corollary}

\newtheorem{remark}[theorem]{Remark}
\newtheorem{prop}[theorem]{Proposition}
\newtheorem{lemma}[theorem]{Lemma}

\newtheorem{example}[theorem]{Example}

\numberwithin{equation}{section}
\numberwithin{theorem}{section}

\begin{document}
\title[General stability for the Pr\'ekopa--Leindler inequality]{A quantitative stability result for the Pr\'ekopa--Leindler inequality for arbitrary measurable functions}

\author{K\'aroly J. B\"or\"oczky, Alessio Figalli and Jo\~ao P. G. Ramos}
\address{Alfr\'ed R\'enyi Institute of Mathematics,
Re\'altanoda street 13-15, H-1053, Budapest, Hungary}
\email{boroczky.karoly.j@renyi.hu}
\address{ETH Z\"urich, Department of Mathematics, R\"amistrasse 101, 8092, Z\"urich, Switzerland}
\email{alessio.figalli@math.ethz.ch}
\address{ETH Z\"urich, Department of Mathematics, R\"amistrasse 101, 8092, Z\"urich, Switzerland}
\email{joao.ramos@math.ethz.ch}

\begin{abstract} We prove that if a triplet of functions satisfies almost equality in the Pr\'ekopa--Leindler inequality, then these functions are close to a common log-concave function, up to multiplication and rescaling.
Our result holds for general measurable functions in all dimensions,
and provides a quantitative stability estimate with computable constants.
\end{abstract}

\maketitle

\section{Introduction} 

\subsection{Brunn-Minkowski and Pr\'ekopa-Leindler inequalities}
Writing $|X|$ to denote Lebesgue measure of a measurable subset $X$ of $\R^n$
(with $|\emptyset|=0$), the Brunn-Minkowski{-Lusternik} inequality 
states that if $\alpha,\beta>0$ and $A,B,C$ are bounded measurable subsets of $\R^n$
with $\alpha A+\beta B \subset C$,\footnote{By convention, if one of the sets $A$ or $B$ is empty, then $\alpha A+\beta B:=\emptyset$.} then
\begin{equation}
\label{BrunnMinkowski}
|C|^{\frac1n}\geq \alpha |A|^{\frac1n}+\beta|B|^{\frac1n}.
\end{equation}
Also, in the case when $|A|>0$ and $|B|> 0$, equality holds if and only if there exist a convex body $K$ (that is, a convex compact set with nonempty interior),
constants $a,b>0$, and vectors $x,y\in\R^n$, such that $\alpha a+\beta b=1$, $\alpha x+\beta y=0$, and
\begin{equation}
\label{BrunnMinkowski-equa}
A\subset a K+x, \mbox{ \ \ }B\subset b K+y, \mbox{ \ \ } |(a K+x)\backslash A|=0,\mbox{ \ \ }
 |(b K+y)\backslash B|=0, \mbox{ \ and \ }
 |K\Delta C|=0,
\end{equation}
where $K\Delta C$ stands for the symmetric difference between $K$ and $C$. 
We note that 
even if $A$ and $B$ are Lebesgue measurable, the Minkowski linear combination 
$\alpha A+\beta B$ may not be measurable (while $\alpha A+\beta B$ is measurable if $A$ and $B$ Borel).
We refer to the monograph \cite{Sch14} for a detailed exposition on this beautiful topic.

\smallskip

The Pr\'ekopa-Leindler inequality is a functional generalization of the classical Brunn-Minkowski inequality.
In order to state it precisely, we recall that a function 
$f: \R^n \to \R_{\ge 0}$ is said to be log-concave if $f\left((1-\lambda)x+\lambda y\right) \ge f(x)^{1-\lambda} f(y)^{\lambda}$  for all $x,y \in \R^n$ and $\lambda\in(0,1)$; in other words, $f$ is log-concave if it can be written as $f=e^{-\varphi}$
for some convex function $\varphi: \R^n \to (-\infty,\infty]$.

\begin{theorem}[Pr\'ekopa, Leindler; Dubuc]
\label{PL}
Let $\lambda\in(0,1)$ and $f,g,h : \R^n \to \R_{\ge 0}$ be measurable functions such that
\begin{equation}\label{eq:condition}
h\left((1-\lambda)x+\lambda y\right) \ge f(x)^{1-\lambda} g(y)^{\lambda} \qquad \forall\,x,y\in \R^n.
\end{equation}
Then
\begin{equation}
\label{PL-formula}
\int_{\R} h \ge \left(\int_{\R} f\right)^{1-\lambda} \left(\int_{\R} g\right)^{\lambda}.
\end{equation}
Also, equality holds if and only if there exist $a > 0$, $w \in \R^n$, and a log-concave function $\tilde{h}$, such that
$h=\tilde{h}$,  $f = a^{-\lambda} \tilde{h}(\cdot - \lambda\,w)$, 
$g = a^{1-\lambda} \tilde{h}(\cdot + (1-\lambda)w)$ almost everywhere. 
\end{theorem}

Note that, if $f,g,h$ are the indicator functions of some sets $A,B,C$, then Theorem~\ref{PL} corresponds exactly to the Brunn-Minkowski inequality.

The Pr\'ekopa-Leindler inequality, due to
Pr\'ekopa \cite{Pre71} and Leindler \cite{Lei72} in dimension one, was generalized in
 Pr\'ekopa \cite{Pre73} and  Borell \cite{Bor75} to any dimension ({\it cf.} Marsiglietti \cite{Mar17}, Bueno, Pivovarov \cite{BuP}, Brascamp, Lieb \cite{BrL76}, Kolesnikov, Werner \cite{KoW}, Bobkov, Colesanti, Fragal\`a \cite{BCF14}). The case of equality is characterized by Dubuc  \cite{Dub77}.
Various applications are provided and surveyed in Gardner \cite{Gar02}.

\subsection{Stability questions}
As discussed above, optimizers are known both for the Brunn-Minkowski and Pr\'ekopa-Leindler inequalities.
However, in spite of knowing the equality cases  for these inequalities, one might ask about what geometric properties can be deduced if one knows that the equality is `almost' attained. This is what one usually refers to as \emph{stability} estimates.

Recently, various important stability results about geometric and functional inequalities have been obtained. For example, Fusco, Maggi, Pratelli \cite{FMP08} proved an optimal stability version of the isoperimetric inequality. This result was extended to the anisotropic isoperimetric inequality and to the Brunn-Minkowski inequality for convex sets by  Figalli, Maggi, Pratelli \cite{FMP09,FMP10} (for the latter problem, the current best estimate is due to Kolesnikov, Milman \cite{KoM}). One can further mention, for instance, stronger versions of the functional Blaschke-Santal\'o inequality, provided by the work of  Barthe, B\"or\"oczky, Fradelizi \cite{BBF14}; of the Borell-Brascamp-Lieb inequality, provided by Ghilli, Salani \cite{GhS17}, Rossi, Salani \cite{RoS17,RoS19} and Balogh, Krist\'aly \cite{BaK18}; of the Sobolev inequality by Figalli, Zhang \cite{FiZ} (extending Bianchi, Egnell \cite{BiE91} and Figalli, Neumayer \cite{FiN19}), 
Nguyen \cite{Ngu16} and Wang \cite{Wan16}; of the log-Sobolev inequality by Gozlan \cite{Goz};
and of some related inequalities by Caglar, Werner \cite{CaW17}, Cordero-Erausquin \cite{Cor17} and Kolesnikov, Kosov \cite{KoK17}. An ``isomorphic" stability result for the Prekopa-Leindler inequality for log-concave functions in terms of the transportation distance has been obtained by Eldan \cite[Lemma~5.2]{Eld13}.

\subsubsection{Stability for Brunn-Minkowski}

About the specific case of the Brunn--Minkowski inequality \eqref{BrunnMinkowski}, the stability question is rather delicate. The first contribution in the direction of stability was made by Freiman \cite{Freiman}, although indirectly, as a consequence of his celebrated $3k-4$ theorem in dimension $n=1$ (see also Christ \cite{Chr12}):

\begin{theorem}[Freiman]
\label{BrunnMinkowskistabn1}
Let $A,B,C\subset\R$ be bounded measurable sets satisfying 
$A+B \subset C$ and $|C|<|A|+|B|+\varepsilon$ for some $\varepsilon \leq \min\{|A|,|B|\}$. Then
there exist intervals $I,J\subset\R$ such that
$A\subset I$, $B\subset J$, $|I\backslash A|<\varepsilon$
and $|J\backslash B|<\varepsilon$.
\end{theorem}

In the planar case, van Hintum, Spink, Tiba \cite{HST22} have found the optimal stability version
 of \eqref{BrunnMinkowski}.
 
\begin{theorem}[van Hintum, Spink, Tiba]
\label{BrunnMinkowskistabn2}
For $\tau\in(0,\frac12]$ and $\lambda\in[\tau,1-\tau]$, let
 $A,B,C$ be bounded measurable subsets of $\R^2$ satisfying
$(1-\lambda)A+\lambda B \subset C$ and
$$
\Big||A|-1\Big|+\Big||B|-1\Big|+\Big||C|-1\Big|<\varepsilon
$$
for some $\varepsilon \leq e^{-M(\tau)}$, with $M(\tau)>0$ depending only on $\tau$.
Then there exists a convex body $K$, {
with $A\subset K+x$ and  $B\subset K+y$  for  some $x,y\in\R^2$,
such that
\begin{equation}
\label{BrunnMinkowskistab00}
|(K+x)\backslash A|+ |(K+y)\backslash B|+|K\Delta C|<c\tau^{-\frac12}\varepsilon^{\frac12}
\end{equation}
for} an absolute constant $c>0$.
\end{theorem}

We note that, for $n\geq 2$, 
in \eqref{BrunnMinkowskistab00} one cannot have an estimate with better error term, both in terms of the order of $\tau$ and $\varepsilon$.
In higher dimensions, the only available quantitative stability version of the Brunn-Minkowski inequality has been established  by Figalli, Jerison \cite{FigalliJerison}.

\begin{theorem}[Figalli, Jerison]
\label{BrunnMinkowskistab}
For $\tau\in(0,\frac12]$ and $\lambda\in[\tau,1-\tau]$, let
 $A,B,C$ be bounded measurable subsets of $\R^n$, $n\geq 3$, with
$(1-\lambda)A+\lambda B \subset C$ and
$$
\Big||A|-1\Big|+\Big||B|-1\Big|+\Big||C|-1\Big|<\varepsilon
$$
for some $\varepsilon<e^{-A_n(\tau)}$,  with $A_n(\tau):=\frac{2^{3^{n+2}}n^{3^n}|\log \tau|^{3^n} }{\tau^{3^n}}$.
Then there exists a convex body $K$,
with $A\subset K+x$ and  $B\subset K+y$  for some $x,y\in\R^n$,
such that
\begin{equation}
\label{BrunnMinkowskistab0}
|(K+x)\backslash A|+
 |(K+y)\backslash B|+
 |K\Delta C|<\tau^{-N_n}\varepsilon^{\gamma_n(\tau)}
\end{equation}
where  $\gamma_n(\tau)=\frac{\tau^{3^n}}{2^{3^{n+1}}n^{3^n}|\log \tau|^{3^n} }$
and $N_n>0$ depends only on $n$.
\end{theorem}

\begin{remark}
We list here some result for particular cases of Theorem~\ref{BrunnMinkowskistab}.
\begin{itemize}
\item When $A=B$, van Hintum, Spink, Tiba \cite{HST21} obtained the optimal stability version, where the error term in \eqref{BrunnMinkowskistab0} is of the from $c_n\tau^{-\frac12}\varepsilon^{\frac12}$
with $c_n>0$ depending only on $n$. Their result improves the previous contributions \cite{FJA,FJABvol,FJA23}.
\end{itemize}
When at least one of the sets $A$ or $B$ is convex, several results have been obtained, as described below.
However, it is important to observe that all these results measure stability by controlling the symmetric difference between $A$ and a translate of $B$. This is weaker than the statement in Theorem~\ref{BrunnMinkowskistab}, where one finds a convex set $K$ that contains both $A$ and $B$ (up to a translation) with a control on the missing volume. Here are some important results.
\begin{itemize}
\item When  either $A$ or $B$ is convex, an optimal stability estimate has been proved by Barchiesi, Julin \cite{BaJ17}.
This extends earlier results  about the case 
where both $A$ and $B$ are convex \cite{FMP09,FMP10}, or when either $A$ or $B$ is the unit ball \cite{FMM}.
\item If $A$ and $B$ are convex and $n$ is large, then Kolesnikov, Milman \cite{KoM} provided an estimate on $|A\Delta(x+B)|$ with a bound of the form $c\,n^{2.75} \tau^{-\frac12}\varepsilon^{\frac12}$,  for some absolute constant $c$. Actually, we note that the term $n^{2.75}$ can be improved to $n^{2.5+o(1)}$ by combining the general estimates of Kolesnikov, Milman \cite[Section 12]{KoM} with the bound $n^{o(1)}$ on the Cheeger constant of a convex body in isotropic position, that follows from Chen’s work \cite{Che21} on the Kannan-Lovasz-Simonovits conjecture.
\end{itemize}
\end{remark}

\subsubsection{Stability for Pr\'ekopa-Leindler}
With respect to the Brunn-Minkowski inequality, until now much less was known about stability for the Pr\'ekopa Leindler inequality, except for some results in the case of log-concave functions (see the discussion below).
In this paper, we prove the first  quantitative stability result for the Pr\'ekopa-Leindler inequality on arbitrary functions.

\begin{theorem}
\label{PLstab}
Given $\tau\in(0,\frac12]$ and $\lambda\in[\tau,1-\tau]$, let $f,g,h : \R^n \to \R_{\ge 0}$ be measurable functions such that $h\left((1-\lambda)x+\lambda y\right) \ge f(x)^{1-\lambda} g(y)^{\lambda}$ for all $x,y \in \R^n$,
and 
\begin{equation}\label{eq:almost-eq}
\int_{\R^n} h < (1+\eps) \left(\int_{\R^n} f\right)^{1-\lambda} \left(\int_{\R^n} g\right)^{\lambda}\qquad \text{for some $\eps>0$.}
\end{equation}
There are a computable dimensional constant $\Theta_n$ and computable constants $Q_n(\tau)$ and $M_n(\tau)$ depending only on $n$ and $\tau$,\footnote{At the end of the proof of Theorem~\ref{PLstab}, we provide explicit values for the constants $M_n(\tau),Q_n(\tau), \Theta_n$.} such that the following holds: If
$0< \varepsilon < e^{-M_n(\tau)}$,
then there exist $\tilde{h}$ log-concave  and $w \in \R^n$ such that 
\[
\int_{\R^n} |h-\tilde{h}| + \int_{\R^n} |a^{\lambda}f-   \tilde{h}(\cdot + \lambda\,w)| + 
\int_{\R^n} |a^{\lambda-1}g- \tilde{h}(\cdot + (\lambda-1)w)| <
\frac{\varepsilon^{Q_n(\tau)}}{\tau^{\Theta_n}} \int_{\R^n}h,
\]
where $a=\int_{\R^n} g/\int_{\R^n} f$. 
\end{theorem}

\begin{remark} If $f,g,h$ are {\it a priori} assumed to be log-concave, then Theorem~\ref{PLstab} was established by Ball, B\"or\"oczky \cite{BallBoroczky} and B\"or\"oczky, De \cite{BoroczkyDe} in the case $n=1$ (in this case, $\varepsilon^{Q_n(\tau)}/\tau^{\Theta_n}$ in Theorem~\ref{PLstab} can be essentially replaced by $(\varepsilon/\tau)^{\frac1{3}}$; see also Theorem~\ref{PLstab-logconv-dim1}),
and by B\"or\"oczky, De \cite{BoroczkyDe} in the case $n\geq 2$ (in that case, 
$\varepsilon^{Q_n(\tau)}/\tau^{\Theta_n}$ in Theorem~\ref{PLstab} can be replaced by
 $(\varepsilon/\tau)^{\frac1{19}}$). Further, we note that Bucur, Fragal\`a \cite{BuF14} proved another interesting stability version of the Pr\'ekopa-Leindler inequality for log-concave functions, bounding the distance of all one dimensional projections.
 \end{remark}

Theorem~\ref{PLstab} is probably quite far from the optimal version, that one could conjecture to provide a bound of the form $C(n,\tau)\varepsilon^{\frac12}$.
In this direction, already for $n=1$,  Example~\ref{bigerrordim1} below shows that
the error term in Theorem~\ref{PLstab} is at least $c\varepsilon^{\frac12}$. 

At first sight, this is perhaps surprising, because in the case of Freiman's result (Theorem~\ref{BrunnMinkowskistabn1}) the error is of order $\varepsilon$, which shows that the Brunn--Minkowski and Pr\'ekopa-Leindler inequalities exhibit different behaviors for $n=1$. Nonetheless, our proof of Theorem~\ref{PLstab} shows that that the Pr\'ekopa--Leindler inequality in dimension $n$ shares some - but not all - of the geometric aspects of the Brunn--Minkowski inequality in dimension $n+1,$ which explains, at least partially, the difference between the two exponents.

Another important difference between the stability version of the Pr\'ekopa-Leindler and the Brunn-Minkowski inequality is shown by the following observation: when $A=B$, the convex set $K$ in Theorem~\ref{BrunnMinkowskistab} coincides with the convex hull of $A$; on the other hand, for $f=g$, the function $\tilde h$ in Theorem \ref{PLstab} can be quite far from the log-concave hull of $f$ (see Example~\ref{examp:no concave hull} below). In other words, there is no direct geometric characterization of the function $\tilde h$ (see also Remark~\ref{rmk:PLvsBM} below).

\smallskip

As mentioned above, the following example shows that
the error term in Theorem~\ref{PLstab} is at least~$c\varepsilon^{\frac12}$.

\begin{example}
\label{bigerrordim1}
There is an absolute constant $c\in (0,1)$ such that the following holds.
For any $\varepsilon\ll 1$, there exist log-concave probability densities $f,g$ on $\R$ such that
\begin{equation}
\label{exa1}
\int_{\R}\sup_{z=\frac12\,x+\frac12\,y}f(x)^{\frac12}g(y)^{\frac12}\,dz<1+\varepsilon,
\end{equation}
while \begin{equation}
\label{exa2}
\int_\R |g(x)-f(x+w)|\,dx\geq c\varepsilon^{\frac12}\qquad \text{for any $w\in\R$}.
\end{equation}
\end{example}
\begin{proof}
We fix $f(x)=e^{-\pi x^2}$ and  an odd $C^2$ function $\varphi$ on $\R$ satisfying
${\rm supp}\,\varphi\subset [-1,1]$ and $\max \varphi=1$. Note that, since $\varphi$ is odd, $\int_{\R}f\varphi=0$.

Given $\eta\ll1$ to be fixed later, we consider
$g=(1+\eta \varphi)f$ so that $\int_{\R}g=1$.
 We note that there exists a constant $\tilde{c}\geq 2$ such that
\begin{eqnarray}
\label{logphiprime}
|[\log(1+\eta\varphi)]'|&=&\left|\eta\cdot \frac{\varphi'}{1+\eta\,\varphi}\right|\leq \tilde{c}\eta\\
\label{logphiprimeprime}
|[\log(1+\eta\varphi)] '' |&=&\left|\eta\cdot \frac{\varphi '' (1+\eta\,\varphi)-\eta(\varphi')^2}{(1+\eta\,\varphi)^2}
\right|\leq \tilde{c}\eta
\end{eqnarray}
for any $\eta\in(0,\frac12)$. In particular, Since $(\log f) '' =-2\pi$, it follows that $g$ is log-concave provided $\eta\ll 1/\tilde{c}$.

Note now that, since $g(x)=f(x)=e^{-\pi x^2}$ for $|x|\geq 1$, there exists a constant $c_0>0$ such that
\begin{equation}
\label{gfw1}
\int_\R |g(x)-f(x+w)|\,dx\geq \int_1^\infty|e^{-\pi x^2}-e^{-\pi (x+w)^2}|\,dx \geq c_0 \min\{|w|,1\}.
\end{equation}
On the other hand, we have
$$
\int_\R |g(x)-f(x+w)|\,dx\geq \int_\R |g(x)-f(x)|-|f(x)-f(x+w)|\,dx\geq
\eta \int_\R f(x)|\varphi(x)|\,dx-\bar c|w|.
$$
Hence, combining this last estimate with \eqref{gfw1}, we deduce the existence of a constant $c_1>0$ such that
\begin{equation}
\label{gfw}
\int_\R |g(x)-f(x+w)|\,dx\geq c_1\eta \qquad \forall\, w \in \R.
\end{equation}

Finally, we estimate $\int_\R h$ for $h(z)=\sup_{2z=x+y}\sqrt{f(x)g(y)}$.
To this aim, consider the auxiliary function
$\tilde{h}(z)=\sqrt{f(z)g(z)}$.
Thanks to H\"older inequality, this satisfies
$\int_{\R}\tilde{h}\leq 1.$ 

Since $f$ and $g$ are log-concave and  $g(x)=f(x)$ for $|x|\geq 1$, 
for any $z\in\R$,
there exists a point $y_z\in \R$  such that $h(z)=\sqrt{f(2z-y_z)g(y_z)}$. Also,
$y_z=z$ if $|z|\geq 1$, and $|y_z|\leq 1$ if $|z|\leq 1$.

We now observe 
that, for any $z\in\R$, the function
$\psi_z(y)=\log\sqrt{f(2z-y)g(y)}$ 
satisfies $\psi_z(z)=\log\tilde{h}(z)$, $\psi_z(y_z)=\log h(z)$ ,
and $\psi_z$ has a maximum at $y_z$. Then, recalling \eqref{logphiprime}, we have
$$
0=\psi'_z(y_z)=2\pi(z-y_z)+\mbox{$\frac12\,[\log(1+\eta\varphi)]'(y_z)$}\qquad\Rightarrow\qquad |z-y_z|\leq \tilde{c}\eta.
$$ 
Hence, since
$|\psi_z ''|$ is bounded, a Taylor expansion yields (recall that $\psi'_z(y_z)=0$)
$$
\log\frac{h(z)}{\tilde{h}(z)}=\psi_z(y_z)-\psi_z(z)\leq c_2\eta^2\qquad \forall\,z\in \R,
$$
for some constant $c_2>1$, and  we conclude that
$$
\int_\R h\leq e^{c_2\eta^2}\int_\R \tilde{h}\leq e^{c_2\eta^2}< 1+2c_2\eta^2 \qquad\text{for }\eta \ll 1.
$$
Choosing $\eta:= (2c_2)^{-\frac{1}{2}}\varepsilon^{\frac12}$, \eqref{gfw} and the equation above prove the result.
\end{proof}

The next examples shows that, even in the case $f=g$,
the function $\tilde h$ provided by Theorem \ref{PLstab} {\it cannot} be chosen to be the log-concave hull of $f$  (i.e., the smallest log-concave function above $f$).
\begin{example}
\label{examp:no concave hull}
For any $\eps>0$ there exist $f,h : \R \to \R_{\ge 0}$ measurable functions such that $h\left(\frac12 x+\frac12 y\right) \ge f(x)^{\frac12}f(y)^{\frac12}$ for all $x,y \in \R^n$,
$$
\int_{\R} h < (1+\eps) \int_{\R} f,
$$
but
$$
\int_\R (F-f)\geq \frac12 \int_\R f,
$$
where $F$ denotes the log-concave hull of $f$.
\end{example}
\begin{proof}
Given $A \gg 1$, let $f$ be defined as
$$
f(x)=\left\{
\begin{array}{ll}
e^{-x} &\text{on }[0,1]\cup [2A,2A+1]\\
0 &\text{otherwise}
\end{array}
\right.
$$
and set $h(z):=\sup_{z=\frac12\,x+\frac12\,y}f(x)^{\frac12}f(y)^{\frac12}$.
Then
$$
h(x)=\left\{
\begin{array}{ll}
e^{-x} &\text{on }[0,1]\cup[A,A+1]\cup [2A,2A+1]\\
0 &\text{otherwise}
\end{array}
\right.
$$
and therefore
$$
\int_{\R} h < (1+\eps) \int_{\R} f
$$
with $\eps \simeq e^{-A} \ll 1$.
On the other hand, the log-concave hull of $f$ is given by
$$
F(x)=\left\{
\begin{array}{ll}
e^{-x} &\text{on }[0,2A+1]\\
0 &\text{otherwise}
\end{array}
\right..
$$
Hence, for $A\gg 1$,
$$
\int_{\R}(F-f)=\int_{1}^{2A}e^{-x}dx =e^{-1}-e^{-2A} \geq \frac12\left(1-e^{-1}\right)=\frac12 \int_\R f ,
$$
as desired.
\end{proof}

\begin{remark}
\label{rmk:PLvsBM}
The argument used in Example \ref{examp:no concave hull} emphasizes a key difference between the Brunn-Minkowski inequality and the Pr\'ekopa-Leindler inequality: while in the Brunn-Minkowski inequality
only arithmetic means of points are considered, in Pr\'ekopa-Leindler one considers points $z$ that are the arithmetic mean of $x$ and $y$, but then the value of $h(z)$ is obtained as a geometric mean of the values of $f(x)$ and $g(y)$. This key difference is the source of many new challenges when proving stability results for Pr\'ekopa-Leindler.
\end{remark}

\subsection{Outline of the proof of Theorem~\ref{PLstab}} We now sketch the structure of the proof of Theorem~\ref{PLstab}, which is split in four main steps.
The first three steps deal with the one-dimensional case. Then, in Step 4, we exploit both the one dimensional case and  Theorem~\ref{BrunnMinkowskistab} to obtain the higher-dimensional result.

\begin{enumerate}
 \item We first deal with the case of symmetrically rearranged functions, and prove the result in this case.
Note that, if $f,g,h$ satisfy  \eqref{eq:condition} and \eqref{eq:almost-eq}, then also their rearrangements $f^*,g^*,h^*$ satisfy the same estimates.

 \item With the knowledge that the result holds for $f^*,g^*,h^*,$ we deduce conditions on the distribution functions $t \mapsto \H^1(\{f>t\}), \, \H^1(\{g>t\})$. In particular, from \eqref{eq:almost-eq} applied to $f,g,h,$ we use a stability version of the 
 Brunn--Minkowski inequality in one-dimension in order to prove that $f$ and $g$ are close to ``bubble-shaped'' functions (i.e., that are nondecreasing on an interval $(-\infty,a)$ and nonincreasing on $(a,+\infty)$).
 
Calling  $\phi$ and $\psi$ such ``bubble-shaped'' functions, we define 
 \[
 \lambda(z) = \sup_{(1-\lambda)x + \lambda y = z} \phi(x)^{1-\lambda} \psi(y)^{\lambda}.
 \] 
This function is measurable (thanks to the fact that $\phi$ and $\psi$ are ``bubble-shaped''), and an analysis similar to the proof of Proposition~\ref{thm:close-stable} shows that $\phi,\psi,\lambda$ satisfy both \eqref{eq:condition} and \eqref{eq:almost-eq} (but for some smaller power of $\eps$). 
 \item Denote
 \begin{equation*}\begin{split}
  \{x \in \R \colon \phi(x) > t\}  = (a_f(t),b_f(t)),\qquad
  \{x \in \R \colon \psi(x) > t\} & = (a_g(t),b_g(t)).
 \end{split}\end{equation*}
Then we use the almost-optimality of $\phi,\psi,\lambda$ to prove that, on a large set, a four-point inequality (in the same spirit of \cite[Lemma~3.6~and~Remark~4.1]{FigalliJerison}) is satisfied by the functions $\mathcal{B}_f(T) = b_f(e^T)$ and $\mathcal{B}_g(T) = b_g(e^T),$ and a `reversed' version 
 of such four-point inequality holds for $\mathcal{A}_f(T) = a_f(e^T)$ and $\mathcal{A}_g(T) = a_g(e^T).$ 

As a consequence, we are able to prove that $\mathcal{A}_f,\mathcal{A}_g$ are both $L^1$-close to convex functions $m_f,m_g$ on a large interval. Analogously, $\mathcal{B}_f,\mathcal{B}_g$ are $L^1$-close to concave functions $n_f,n_g$ on the same large interval. 
Thanks to these facts, we show that there exist log-concave function $\tilde{\phi}$ and $\tilde\psi$ such that $\{\tilde{\phi} > t\} = (m_f(\log t), n_f(\log t))$ and $\{\tilde{\psi} > t\} = (m_g(\log t), n_g(\log t))$ on a large interval. 

 Finally, we translate the properties of $\mathcal{A}_f,\mathcal{A}_g,\mathcal{B}_f,\mathcal{B}_g,m_f,m_g,n_f,n_g$ into a bound on $\|\phi-\tilde{\phi}\|_1,$ which can be thus made small. By Proposition~\ref{thm:close-stable}, we conclude the one-dimensional case of Theorem~\ref{PLstab}.
 
 \item In order to obtain the result also in higher dimensions, we consider the hypographs of the logarithms of $f,g,h$. Denoting these sets by $\mathcal{S}_f,\mathcal{S}_g,\mathcal{S}_h,$ respectively, we show that they satisfy the Brunn--Minkowski condition $\mathcal{S}_h \supset (1-\lambda)\mathcal{S}_f + \lambda \mathcal{S}_g.$ In particular, due to the one-dimensional case, we can estimate how level sets of $f,g,h$ are close to each other, in terms of volume. This enables us to use the main theorem in \cite{FigalliJerison} on the sets $\mathcal{S}_f,\mathcal{S}_g,\mathcal{S}_h,$ which in turn produces a natural algorithm to construct log-concave functions  close to $f,g,h.$ 
\end{enumerate}

The rest of the manuscript is organized as follows: in Section~\ref{sec:tail-estimates}, we prove tail estimates that allow us to suitably truncate the functions under consideration, as well as estimate on the size of level sets. This allows us to perform a set of preliminary reductions of the one-dimensional problem. In Section~\ref{sec:rearranged-version}, we prove Theorem~\ref{PLstab} in the case when $n=1$ and $f,g,h$ are symmetrically decreasing, while in Section~\ref{sec:general-case} we deal with the general one dimensional case. Finally, in Section~\ref{sec:hi-d-case}, we prove the theorem in arbitrary dimension. 

\vspace{2mm}

Throughout the manuscript, we will use the notation $\H^k$ for the $k-$dimensional Hausdorff measure of a set. Sometimes we shall use $c>0$ to denote an absolute (computable) constant, whose exact value might change from one part of the paper to the next, and even from line to line. We will also occasionally use a subscript, e.g. $c_n$, to indicate dependence of the constant on a dimensional parameter. Moreover, we write $a \lesssim b$ whenever $a/b$ is bounded from above by an absolute and explicitly computable constant, and we shall use a subscript $a \lesssim_n b$ to emphasize the dependence of the bound on the dimension considered. Finally, we write $a \simeq b$ if both $a \lesssim b$ and $b \lesssim a$ hold.

\medskip

{\noindent \it Acknowledgments.}
The first author is supported by the NKFIH Grant 132002.
The second and third author are supported by the European Research Council under the Grant Agreement No.
721675 ``Regularity and Stability in Partial Differential Equations (RSPDE).''

\section{Tail estimates in the case of almost equality in the one-dimensional Pr\'ekopa-Leindler inequality}\label{sec:tail-estimates}

A useful tool for our study is the symmetric decreasing rearrangement. For a bounded function $\varphi:\R\to \R_{\geq 0}$ with $0<\int_\R\varphi<\infty$, 
we define its symmetric decreasing rearrangment $\varphi^*:\R\to \R_{\geq 0}$ by
$$
\varphi^*(t)=\inf\big\{\alpha:\,\H^1(\{\varphi\geq \alpha\})\leq 2|t|\big\}.
$$
In particular, $\varphi^*$ is an even function that is monotone decreasing on $[0,\infty)$, $\varphi^*(0)$ is the essential supremum of 
$\varphi,$ and 
\begin{equation}
\label{pfi*}
\H^1(\{\varphi\geq \alpha\})=\H^1(\{\varphi^* \geq \alpha\})
\end{equation}
for any $\alpha>0$ with $\H^1(\{\varphi\geq \alpha\})>0$. In particular, the level sets $\{\varphi^*\geq \alpha\}$ are symmetric segments, and the layer cake representation
yields $\int_\R\varphi=\int_\R\varphi^*$.

Symmetric decreasing rearrangement works very well for the Pr\'ekopa-Leindler inequality. For $\lambda\in(0,1)$ and
bounded functions $f,g,h:\R\to \R_{\geq 0}$ with positive integral,
if $h((1-\lambda)x+\lambda y)\geq f(x)^{1-\lambda}g(y)^\lambda$  for any $x,y\in\R$, then
the one-dimensional Brunn-Minkowski inequality yields 
$h^*((1-\lambda)x+\lambda y)\geq f^*(x)^{1-\lambda}g^*(y)^\lambda$ for any $x,y\in\R$.
Also, if $\varphi$ is log-concave, then the same holds for $\varphi^*$.

The main goal of this section is to show that if we have almost equality in the one-dimensional Pr\'ekopa-Leindler equality, then the functions $f,g,h$ in \eqref{eq:almost-eq} with positive integral satisfy similar tail estimates like log-concave functions
(here $\varphi: \R \to \R_{\ge 0}$ has positive integral if
$0<\int \varphi<\infty$). 
First we review the related properties of log-concave functions. Let us recall the following estimate from  
 \cite{BallBoroczky,BoroczkyDe}:
 
\begin{theorem}[Ball, B\"or\"oczky, De]
\label{PLstab-logconv-dim1}
For $\tau\in(0,\frac12]$ and $\lambda\in[\tau,1-\tau]$, let $f,g,h : \R \to \R_{\ge 0}$ be log-concave functions with positive integral such that $h\left((1-\lambda)x+\lambda y\right) \ge f(x)^{1-\lambda} g(y)^{\lambda}$ 
for all $x,y \in \R$, and
\begin{equation}\label{PLstab-logconv-dim1-eq}
\int_{\R} h < (1+\eps) \left(\int_{\R} f\right)^{1-\lambda} \left(\int_{\R} g\right)^{\lambda}
\end{equation}
for some $\varepsilon\in(0,1)$.
Then there exists $w \in \R$ such that
\[
 \int_{{\R}} |a^{\lambda}f-   h(\cdot +\lambda\,w)| + 
\int_{{\R}} |a^{\lambda-1}g- h(\cdot + (\lambda-1)w)| < c
\left(\frac{\varepsilon}{\tau}\right)^{\frac13}|\log \varepsilon|^{\frac43}\int_{\R^n}h,
\]
where  $a=\int_{{\R}} g/\int_{{\R}} f$,  and $c>1$ is an absolute constant.
\end{theorem}

Next, we prove some basic properties of log-concave functions. We observe that 
if $\varphi$ is log-concave and $0<\int_\R\varphi<\infty$, then the level sets are segments, $\varphi$ is bounded, and its essential supremum coincides with its supremum $\|\varphi\|_\infty$.

\begin{lemma}
\label{logconvdim1}
Let $\varphi$ be a log-concave function with $0<\int_\R\varphi<\infty$. Then:
\begin{description}
\item[(i)] $\H^1(\{\varphi>\|\varphi\|_\infty-s\})\geq \frac{\|\varphi\|_1}{\|\varphi\|_\infty^2} s$
\ \ provided $0<s<\|\varphi\|_\infty$;
\item[(ii)] $\H^1(\{\varphi>t\}) \leq \frac{2\|\varphi\|_1}{\|\varphi\|_\infty}\left|\log\frac{t}{\|\varphi\|_\infty}\right|$
\ \ provided $0<t\leq\frac12\,\|\varphi\|_\infty$;
\item[(iii)] $\int_{\{\varphi<t\}}\varphi \leq \frac{2\|\varphi\|_1}{\|\varphi\|_\infty} t$
\ \ provided $0<t\leq\frac12\,\|\varphi\|_\infty$.
\end{description}
\end{lemma}
\begin{proof}
Using symmetric decreasing rearrangement we can assume that $\varphi$ is even.
Also, by scaling, we may also suppose that $\varphi(0)=\|\varphi\|_\infty=\int_\R\varphi=1$.

For (i), let $x_0=\sup\{x:\,\varphi(x)>1-s\}=\frac12\,\H^1(\{\varphi>1-s\})$, and choose $\gamma>0$ such that
 $1-s=e^{-\gamma\,x_0}$. It follows from the log-concavity and the evenness of $\varphi$ that 
$\varphi(x)\leq 1$ if $|x|\leq |x_0|$, and 
$\varphi(x)\leq e^{-\gamma\,|x|}$ if $|x|\geq |x_0|$.
Also, since $e^{-\gamma\,x_0}>1-\gamma\,x_0$ we get $\frac1{\gamma}< \frac{x_0}{s}$, thus
$$
1= \int_\R\varphi\leq
2x_0+
2\int_{x_0}^\infty e^{-\gamma\,x}\,dx=2x_0+\frac{2\,e^{-\gamma\,x_0}}{\gamma}
< 2x_0\left(1+\frac{1-s}{s}\right)=\frac{2x_0}{s}.
$$

For (ii) and (iii), let  $x_1=\sup\{x:\,\varphi(x)>t\}=\frac12\,\H^1(\{\varphi>t\})$, and choose $\delta>0$ such that
 $t=e^{-\delta\,x_1}$. It follows again by log-concavity and evenness that
$\varphi(x)\geq e^{-\delta\,|x|}$ if $|x|\leq |x_1|$, and 
$\varphi(x)\leq e^{-\delta\,|x|}$ if $|x|\geq |x_1|$.

Then, on the one hand, we have
\begin{equation}
\label{x1logt}
\frac12\geq \int_0^{x_1} e^{-\delta\,x}\,dx=\frac{1-e^{-\delta\,x_1}}{\delta}=\frac{1-t}{\delta}\geq 
\frac{1}{2\delta}=\frac{x_1}{2|\log t|},
\end{equation}
verifying (ii). On the other hand, using \eqref{x1logt} we get
$$
\int_{\{\varphi<t\}}\varphi\leq 
2 \int_{x_1}^\infty e^{-\delta\,x}\,dx=\frac{2e^{-\delta\,x_1}}{\delta}=\frac{2tx_1}{|\log t|}\leq 2t,
$$
verifying (iii).
\end{proof}

Given $\varepsilon\in(0,1]$, $\tau\in(0,\frac12]$, and $\lambda\in[\tau,1-\tau]$,
we now consider measurable functions $f,g,h:\,\R\to\R_{\geq 0}$ with positive integral 
satisfying
\begin{eqnarray}
\label{fghepscond}
h((1-\lambda)x+\lambda\,y)&\geq&f(x)^{1-\lambda}g(y)^\lambda\mbox{ \ \ for }x,y\in\R\\
\label{fgheps}
\int_\R h&<& (1+\varepsilon)
\left(\int_\R f\right)^{1-\lambda}\left(\int_\R g\right)^{\lambda}.
\end{eqnarray}
For $t>0$, we set
\begin{equation}
\label{ABCt}
A_t=\{f\geq t\},\qquad B_t=\{g\geq t\},\quad \mbox{ and }\quad C_t=\{h\geq t\},
\end{equation}
so that
$$
A_t=\bigcap_{0<s<t}A_s,\qquad B_t=\bigcap_{0<s<t}B_s,\quad \mbox{ and }\quad C_t=\bigcap_{0<s<t}C_s.
$$
It follows from \eqref{fghepscond} that if $A_t,B_s\neq\emptyset$ for $t,s>0$, then
\begin{equation}
\label{ABCtlambda}
(1-\lambda)A_t+\lambda\,B_s\subset C_{t^{1-\lambda}s^{\lambda}}.
\end{equation}

\begin{lemma}
\label{fgbounded}
Let $f,g,h$ satisfy \eqref{fghepscond} and \eqref{fgheps}. Then $f$ and $g$ are bounded.
\end{lemma}
\begin{proof}
For any $x_0\in \R$ with $f(x_0)>0$, we have
$$
2\left(\int_\R f\right)^{1-\lambda}\left(\int_\R g\right)^{\lambda}>\int_\R h\geq
\int_\R f(x_0)^{1-\lambda}g\mbox{$\left(\frac{1}{\lambda}\,z-\frac{1-\lambda}{\lambda}\,x_0\right)^\lambda$}\,dz
=f(x_0)^{1-\lambda}\lambda\int_\R g^\lambda;
$$
therefore, $f$ is bounded. Similarly, $g$ is bounded, as well.
\end{proof}

We use the following stability version of the inequality between the arithmetic and geometric mean. It follows from 
Lemma~2.1 in Aldaz \cite{Ald08} that if $a,b>0$ and $\lambda\in[\tau,1-\tau]$ for $\tau\in(0,\frac12]$, then
\begin{equation}
\label{Holderstab}
(1-\lambda)a+\lambda b-a^{1-\lambda}b^\lambda\geq \tau\left(\sqrt{a}-\sqrt{b}\right)^2.
\end{equation}
According to Lemma~\ref{fgbounded}, we can speak about $\|f\|_{\infty}$ and $\|g\|_{\infty}$.

\begin{lemma}
\label{thm:control-sup} 
Let $f,g,h$ satisfy
\eqref{fghepscond} and \eqref{fgheps}. If $\varepsilon<2^{-6}\tau^3$, then  
\[
\left|\frac{\|f\|_{\infty}}{\|g\|_{\infty}}\cdot \frac{\|g\|_{1}}{\|f\|_{1}} - 1\right| \leq 4\tau^{-\frac32}\varepsilon^{\frac12}.
\]
\end{lemma}
\begin{proof}  We may assume that $\int_\R f=\int_\R g=1$.

We set $\theta=\|f\|_{\infty}/\|g\|_{\infty}$.
Using the notation \eqref{ABCt}, it follows from \eqref{fghepscond} that if $0<t<\|f\|_{\infty}^{1-\lambda}\|g\|_{\infty}^\lambda$, then
$$
(1-\lambda) A_{\theta^{\lambda}t}+\lambda\, B_{\theta^{\lambda-1}t}\subset C_t.
$$
We deduce from \eqref{ABCtlambda} and the one-dimensional Brunn-Minkowski inequality that
\begin{eqnarray*}
1+\varepsilon&\geq&\int_\R h\geq \int_0^{\|f\|_{\infty}^{1-\lambda}\|g\|_{\infty}^\lambda} \H^1(C_t)\,dt\\
&\geq&
(1-\lambda) \int_0^{\|f\|_{\infty}^{1-\lambda}\|g\|_{\infty}^\lambda} \H^1(A_{\theta^{\lambda}t})\,dt+
\lambda \int_0^{\|f\|_{\infty}^{1-\lambda}\|g\|_{\infty}^\lambda} \H^1(B_{\theta^{\lambda-1}t})\,dt\\
&=&\frac{1-\lambda}{\theta^{\lambda}} \int_0^{\|f\|_{\infty}} \H^1(A_{s})\,ds+
\lambda\,\theta^{1-\lambda} \int_0^{\|g\|_{\infty}} \H^1(B_{s})\,ds
=\frac{1-\lambda}{\theta^{\lambda}}+\lambda\,\theta^{1-\lambda}.
\end{eqnarray*}
We conclude from \eqref{Holderstab} that
$$
\left|\theta^{-\frac{\lambda}2}- \theta^{\frac{1-\lambda}2}\right|<\tau^{-\frac12}\varepsilon^{\frac12},
$$
which in turn yields that
$$
\tau^{-\frac12}\varepsilon^{\frac12}>e^{\frac{\tau|\log\theta|}2}-1>\frac{\tau|\log\theta|}2.
$$
Since $|\log\theta|<2\tau^{-\frac32}\varepsilon^{\frac12}\leq \frac14$ provided $\varepsilon\leq \tau^3/64$, 
we have 
$|\theta-1|<4\tau^{-\frac32}\varepsilon^{\frac12}$.
\end{proof}

\begin{lemma}
\label{thm:cutting-support} 
Let $f,g,h$ satisfy
\eqref{fghepscond} and \eqref{fgheps}. If $\varepsilon^{\frac12}\leq \eta<1$, then  
\[
\H^1(\{f\geq \eta\|f\|_\infty\})\lesssim  \frac{\tau^{-\frac52}\|f\|_1}{\|f\|_\infty}\cdot|\log \varepsilon|^{\frac4\tau} ,\qquad
\H^1(\{g \geq \eta\|g\|_\infty\}) \lesssim  \frac{\tau^{-\frac52}\|g\|_1}{\|g\|_\infty}\cdot|\log \varepsilon|^{\frac4\tau},
\]
and 
\[
\int_{\{f < \eta\}} f
\lesssim \tau^{-\frac52}\|f\|_1\cdot\eta\,|\log \varepsilon|^{\frac4\tau},\qquad \int_{\{g <\eta\}} g \lesssim \tau^{-\frac52}\|g\|_1\cdot\eta\,|\log \varepsilon|^{\frac4\tau}.
\]
\end{lemma}

\begin{proof} 
We may assume that $\|f\|_{\infty}=\|g\|_{\infty} = 1$ and $\min\{\int_\R f,\int_\R g\}=1$, so that
Lemma~\ref{thm:control-sup} yields 
\begin{equation}
\label{fgintest-cutting-support}
1=\min\left\{\int_\R f,\int_\R g\right\}\leq \max\left\{\int_\R f,\int_\R g\right\}\leq 1+4\tau^{-\frac32}\varepsilon^{\frac12}<2.
\end{equation}
For $t > 0$, it follows from \eqref{ABCtlambda} that if $\varrho\in(0,1)$, then
\begin{equation}\label{eq:power-containment-levels}
C_{\varrho\,t} 
\supset \left((1-\lambda)A_{t^{\frac{1}{1-\lambda}}} + \lambda\,B_{\varrho^{\frac1{\lambda}}}\right) \cup
\left((1-\lambda)A_{\varrho^{\frac{1}{1-\lambda}}}  + \lambda\,B_{t^{\frac{1}{\lambda}}}\right),
\end{equation}
thus letting $\varrho$ tending to $1$, the one-dimensional  Brunn-Minkowski inequality
 yields
\begin{equation}
\label{eq:contained-1}
\H^1(C_t)\geq \frac12\,\left[(1-\lambda)\H^1\left(A_{t^{\frac{1}{1-\lambda}}}\right) 
+\lambda\,\H^1\left(B_{t^{\frac{1}{\lambda}}}\right)\right].
\end{equation}
 In addition, $\H^1(C_t) - (1-\lambda)\H^1(A_t) -\lambda\,\H^1(B_t)\geq 0$ 
holds for any $t > 0$, thanks to \eqref{ABCtlambda} and the one-dimensional  Brunn-Minkowski inequality.

Therefore, using the near optimality  \eqref{fgheps} for the  Pr\'ekopa-Leindler inequality, \eqref{fgintest-cutting-support}, and \eqref{eq:contained-1},
we deduce that for any $\alpha\in(0,1]$, we have 
\begin{equation}
\label{eq:tau}
\begin{split}
8\tau^{-\frac{3}2}\varepsilon^{\frac12}
 &\ge \int_0^{\alpha} \left( \H^1(C_t) - (1-\lambda)\H^1(A_t) -\lambda\,\H^1(B_t)\right) \, d t \cr 
     &\ge \int_0^{\alpha} \left( \frac12\,\left[(1-\lambda)\H^1\left(A_{t^{\frac{1}{1-\lambda}}}\right) 
+\lambda\,\H^1\left(B_{t^{\frac{1}{\lambda}}}\right)\right]
 -(1-\lambda)\H^1(A_t) -\lambda\,\H^1(B_t)\right) \, d t.
\end{split}
\end{equation}
We now define 
$$
\Gamma(\alpha) := \int_0^{\alpha}\big((1-\lambda)\H^1(A_t) + \lambda\,\H^1(B_t)\big) \, d t.
$$ 
Note that $\Gamma$ is an increasing function bounded by $2$. Also, through a change of variables, it satisfies
\[
\int_0^{\alpha} \big((1-\lambda)\H^1(A_{t^{\frac1s}}) + \lambda\,\H^1(B_{t^{\frac1s}})\big) \, d t \ge 
s\alpha^{1-\frac1{s}} \Gamma(\alpha^{\frac1s})\qquad \forall\,s \in (0,1).
\]
Hence, assuming with no loss of generality that $\lambda \leq 1/2$, it follows from \eqref{eq:tau} that
\begin{equation}\label{eq:iterate}
8\tau^{-\frac{3}2}\varepsilon^{\frac12}
\geq \frac{1-\lambda}2\cdot \alpha^{-\frac\lambda{1-\lambda}}\, \Gamma(\alpha^{\frac{1}{1-\lambda}})- \Gamma(\alpha).
\end{equation}
As $1-\lambda\geq 1/2$, using the substitution $\beta=\alpha^{\frac{1}{1-\lambda}} \in (0,1)$, \eqref{eq:iterate} leads to   
$$
\frac{\Gamma(\beta)}{\beta} \le \frac{32\tau^{-\frac32}\eps^{\frac12}}{\beta^{1-\lambda}} + 
4\frac{\Gamma(\beta^{1-\lambda})}{\beta^{1-\lambda}},
$$
and, by iteration,
\begin{equation}\label{eq:iterate-2}
\frac{\Gamma(\beta)}{\beta} \le 32\tau^{-\frac32}\eps^{\frac12} \sum_{i=1}^k\frac{4^{i-1}}{\beta^{(1-\lambda)^{i}}} + 
4^k\frac{\Gamma(\beta^{(1-\lambda)^k})}{\beta^{(1-\lambda)^k}} \leq c\bigg(1+\tau^{-\frac32}\frac{\eps^{\frac12}}{\beta^{1-\lambda}}\biggr)\frac{4^k}{\beta^{(1-\lambda)^k}}\qquad \forall\, k\geq 1.
\end{equation}
Hence, if $\varepsilon^{\frac12}\leq \beta$, then \eqref{eq:iterate-2} yields
$$
\frac{\Gamma(\beta)}{\beta} \le c\tau^{-\frac32}\frac{4^k}{\beta^{(1-\lambda)^k}}.
$$
Choosing $k \in \left[\frac{|\log|\log\beta||}{|\log(1-\lambda)|},2\frac{|\log|\log\beta||}{|\log(1-\lambda)|}\right]$ so that $\beta^{(1-\lambda)^k}\simeq 1$, then the bound above gives (recall that $\lambda \geq \tau$ and that $|\log(1-\tau)|\simeq \tau$)
$$
\frac{\Gamma(\beta)}{\beta} \le c\tau^{-\frac32}4^{2\frac{|\log|\log\beta||}{\tau}}\leq c \tau^{-\frac32}|\log\beta|^{\frac4\tau}\qquad \forall\,\beta \in [\varepsilon^{\frac12},1).
$$
Since
$$
\frac{\Gamma(\beta)}{\beta}  \geq (1-\lambda)\H^1(A_\beta) + \lambda\,\H^1(B_\beta)\geq \tau \big( \H^1(A_t) + \H^1(B_t)\big),
$$
this proves the first part of the statement of the Lemma.

Finally, the layer cake formula yields $\int_{\{f<\eta\}}f+\int_{\{g<\eta\}}g\leq \Gamma(\eta)$, and the monotinicity of $A_t$ and $B_t$
imply $\H^1(\{f\geq \eta\})+\H^1(\{g\geq \eta\})\leq \Gamma(\eta)/\eta$, completing the proof of
Lemma~\ref{thm:cutting-support}.
\end{proof}

\begin{prop}
\label{thm:close-stable}
Let $f,g,h$ satisfy
\eqref{fghepscond} and \eqref{fgheps}. Let $\eta \geq \eps$, assume that
$\tau^{-\frac32} \eta^{\frac12}\leq c$
for certain absolute constant $c\in(0,1)$, and that there exist 
log-concave functions $\tilde{f},\tilde{g}$ such that 
\[
\|f-\tilde{f}\|_1 < \eta\|f\|_1\mbox{ \ and \ } \|g-\tilde{g}\|_1 < \eta\|g\|_1.
\]
Then, setting $a=\int_\R g/\int_\R f$, there exist a log-concave function $\tilde{h}$ and a constant $w\in\R$ such that
\begin{eqnarray*}
\int_{\R} |a^\lambda f(x) - \tilde{h}(x-\lambda\,w)| \, d x + 
\int_{\R} |a^{\lambda-1}g(x) -\tilde{h}(x+(1-\lambda)w)|\,dx
&\lesssim &\tau^{-1}\eta^{\frac1{12}}|\log\varepsilon|^{\frac43}
\int_\R h,\\
\int_{\R} |h(x) - \tilde{h}(x)| \, d x &\lesssim& \tau^{-2} \eta^{\frac14}|\log \varepsilon|
\int_\R h.
\end{eqnarray*}
\end{prop}

\begin{proof} 
We may assume that $\min\{\|f\|_{\infty},\|g\|_{\infty}\} = 1$ and $\int_\R f=\int_\R g=1$, so that
Lemma~\ref{thm:control-sup} yields 
\begin{equation}
\label{fgintest-close-stable}
1=\min\left\{\|f\|_{\infty},\|g\|_{\infty}\right\}\leq \max\left\{\|f\|_{\infty},\|g\|_{\infty}\right\}\leq 
1+4\tau^{-\frac32}\eps^{\frac12}<2.
\end{equation}
In particular, the approximating log-concave functions satisfy
\begin{equation}
\label{tildefg-close-stable}
\frac12<\int_\R\tilde{f},\int_\R\tilde{g}<2.
\end{equation}
The main idea of the proof is to show that,
for a suitable log-concave function $\tilde{h}$, the log-concave functions
 $\tilde{f}_0=\tilde{f}\chi_{\{\tilde{f}>\alpha\}}$ and  
$\tilde{g}_0=\tilde{g}\chi_{\{\tilde{g}>\alpha\}}$ satisfy almost equality in the Pr\'ekopa-Leindler inequality for
some value $\alpha\geq \eta$; 
therefore, the stability version Theorem~\ref{PLstab-logconv-dim1} of the Pr\'ekopa-Leindler inequality for log-concave functions implies that $\tilde{f}_0$ and $\tilde{g}_0$ can be expressed in terms of shifts and multiples of $\tilde{h}$. 

As a first step, we claim that
\begin{equation}
\label{logconvmaxclose}
|\|\tilde{f}\|_\infty-\|f\|_\infty|\leq 32\tau^{-\frac32} \eta^{\frac12}\mbox{ \ and \ }|\|\tilde{g}\|_\infty-\|g\|_\infty|
\leq 32\tau^{-\frac32} \eta^{\frac12}. 
\end{equation}
As the roles of $f$ and $g$ are symmetric, we only prove the statement about $f$.

First, we assume that $\|\tilde{f}\|_\infty>\|f\|_\infty$, hence $\|f\|_\infty=\|\tilde{f}\|_\infty-\alpha$ for some $\alpha>0$. 
In this case, Lemma~\ref{logconvdim1} (i) and \eqref{tildefg-close-stable}
imply that $\H^1(\{\tilde{f} > \|\tilde f\|_\infty-s\})\geq \frac{s}{2}\,\|\tilde{f}\|_\infty^{-2}$ for
$s\in(0,\alpha)$, thus the layer-cake representation gives
$$
\eta\geq\int_{\|f\|_\infty}^{\|\tilde{f}\|_\infty}\H^1(\{\tilde{f} > t\})\,dt>\frac{\alpha^2}{4\|\tilde{f}\|_\infty^2}.
$$
Therefore $\|f\|_\infty=\|\tilde{f}\|_\infty-\alpha\geq\|\tilde{f}\|_\infty(1-2\sqrt{\eta})$, and we deduce that
$$
\|\tilde{f}\|_\infty-\|f\|_\infty\leq \|f\|_\infty\left[(1-2\sqrt{\eta})^{-1}-1\right]<8\eta^{\frac12}.
$$

Next we assume that $\|\tilde{f}\|_\infty<\|f\|_\infty$. We consider the function
$$
f_1=f \cdot \chi_{\{f \le \|\tilde{f}\|_{\infty}\}}  + \|\tilde{f}\|_{\infty}\cdot \chi_{\{f > \|\tilde{f}\|_{\infty}\}},
$$
that satisfies 
$$
1\leq \left(\int_\R f_1\right)^{-1}\leq \left(\int_\R f-\int_\R |f-\tilde{f}|\right)^{-1}<1+2\eta.
$$
As $f_1\leq f$, we have $h((1-\lambda)x+\lambda\,y)\geq f_1(x)^{1-\lambda}g(y)^\lambda$ for any $x,y\in\R$ where
$$
\int_\R h\leq (1+\varepsilon)\left(\int_\R f\right)^{1-\lambda}\left(\int_\R g\right)^{\lambda}
\leq (1+4\eta)\left (\int_\R f_1\right)^{1-\lambda}\left(\int_\R g\right)^{\lambda}.
$$
We deduce from Lemma~\ref{thm:control-sup} 
applied to $f$ and $g$ on the one hand, and to $f_1$ and $g$ on the other hand that
$$
\frac{\|f\|_{\infty}}{\|\tilde{f}\|_{\infty}}=\frac{\|f\|_{\infty}}{\|g\|_{\infty}}\cdot \frac{\|g\|_{\infty}}{\|f_1\|_{\infty}}
\leq  \left(1+4\tau^{-\frac32}\varepsilon^{\frac12}\right)\cdot \left(1+4\tau^{-\frac32}\eta^{\frac12}\right)(1+4\eta)
<1+16\tau^{-\frac32}\eta^{\frac12}.
$$
Recalling \eqref{fgintest-close-stable}, this proves
the claim \eqref{logconvmaxclose}. In turn, combining 
\eqref{fgintest-close-stable} and \eqref{logconvmaxclose} leads to
\begin{equation}
\label{allmax:control-sup}
\frac12< \|f\|_\infty,\;\|g\|_\infty,\;\|\tilde{f}\|_\infty,\;\|\tilde{g}\|_\infty 
<2.
\end{equation}

For any $r>0$, we define
$$
A_r=\{f>r\},\mbox{ \ }\widetilde{A}_r=\{\tilde{f}>r\},\mbox{ \ }
B_r=\{g>r\},\mbox{ \ }\widetilde{B}_r=\{\tilde{g}>r\}.
$$
According to the layer-cake representation, 
 \begin{eqnarray*}
 \int_0^{\infty} \H^1(A_r \Delta \widetilde{A}_r) \, d r=\|f-\tilde{f}\|_1
&\leq & \eta\\
 \int_0^{\infty} \H^1(B_r \Delta \widetilde{B}_r) \, d r =\|g-\tilde{g}\|_1
&\leq & \eta.
 \end{eqnarray*}
In particular, the set 
$S\subset(0,\infty)$ 
defined by the property
\begin{equation}
\H^1(A_r \Delta \widetilde{A}_r)+\H^1(B_r \Delta \widetilde{B}_r)\leq
 \eta^{\frac12}\mbox{ \ for $r\in S$}
\end{equation}
satisfies that
\begin{equation}
\label{SdefH1:control-sup}
\H^1((0,\infty)\backslash S)< 4\eta^{\frac12}.
\end{equation}
In particular, if $r,s\in S$, then any $x\in\R$ satisfies
\begin{equation}
\label{SdefDelta:control-sup}
\left|\H^1\left((1-\lambda)A_r \cap (x-\lambda\,B_s)\right)-
\H^1\left((1-\lambda)\widetilde{A}_r \cap(x-\lambda\, \widetilde{B}_s)\right)\right|\leq
 \eta^{\frac12}.
\end{equation}
Consider
\begin{equation}
\label{r0s0def:control-sup}
r_0=\|\tilde{f}\|_\infty-32\tau^{-1}\eta^{\frac14}\quad
\mbox{ \ and \ }\quad s_0=\|\tilde{g}\|_\infty-32\tau^{-1}\eta^{\frac14}.
\end{equation}
Using \eqref{tildefg-close-stable} and \eqref{allmax:control-sup}, we deduce from Lemma~\ref{logconvdim1} (i)
that
\begin{equation}
\label{r0:control-sup}
\H^1(\widetilde{A}_{r_0}),\;
\H^1(\widetilde{B}_{s_0})\geq  4\tau^{-1}\eta^{\frac14}.
\end{equation}
Possibly after shifting $f$ and $g$, we may assume that zero is the common midpoint of the segments
$\widetilde{A}_{r_0}$ and $\widetilde{B}_{s_0}$. In particular, setting 
$$
{\rm cl}\,\widetilde{A}_r=[a_1(r),a_2(r)]\mbox{ \ and \ }{\rm cl}\,\widetilde{B}_s=[b_1(s),b_2(s)]\qquad \text{for $0<r<\|\tilde{f}\|_\infty$ and $0<s<\|\tilde{g}\|_\infty$},
$$
using that $a_1(r),b_1(r)$ are monotone decreasing and $a_2(r),b_2(r)$ are monotone increasing
provided $0<r<\min\{\|\tilde{f}\|_\infty,\|\tilde{g}\|_\infty\}$,
we have
$$
a_2(r),b_2(s)\geq 2\tau^{-1}\eta^{\frac14}\mbox{ \ and \ }a_1(r),b_1(s)\leq -2\tau^{-1}\eta^{\frac14}
\mbox{ \ for $r\in(0,r_0]$, $s \in (0,s_0]$}.
$$
Therefore, if $r\in S\cap (0,r_0)$, $s\in S\cap (0,s_0)$ and
$$
 x\in (1+2\eta^{\frac14})^{-1}\left((1-\lambda)\widetilde{A}_r + (\lambda\, \widetilde{B}_s)\right)
\subset (1-\eta^{\frac14})\left((1-\lambda)\widetilde{A}_r + (\lambda\, \widetilde{B}_s)\right),
$$
then 
$$
\H^1\left((1-\lambda)\widetilde{A}_r \cap(x-\lambda\, \widetilde{B}_s)\right)\geq 2\eta^{\frac12},
$$
and hence \eqref{SdefDelta:control-sup} yields 
$$
x\in(1-\lambda)A_r + (\lambda\,B_s).
$$
In other words, if $r\in S\cap (0,r_0)$ and $s\in S\cap (0,s_0)$, then
\begin{equation}
\label{rsinSrs0:control-sup}
(1-\lambda)\widetilde{A}_r + \lambda\, \widetilde{B}_s\subset
 (1+2\eta^{\frac14})\left((1-\lambda)A_r +\lambda\,B_s\right)\subset
(1+2\eta^{\frac14})\left\{h>r^{1-\lambda}s^\lambda\right\}.
\end{equation}
On the other hand, for any $r\in (\eta^{\frac14},\|\tilde{f}\|_\infty)$ and $s\in (\eta^{\frac14},\|\tilde{g}\|_\infty)$,
\eqref{SdefH1:control-sup} and the definition
of $r_0,s_0$ yield the existence of some
$\tilde{r}\in S\cap (0,\min\{r,r_0\})$ and $\tilde{s}\in S\cap (0,\min\{s,s_0\})$ with
$$
\tilde{r}\geq r-\theta(r)\mbox{ \ and \ }\tilde{s}\geq s-\theta(s)
$$
where $\theta(t)=2^6\tau^{-1}\eta^{\frac14}$ if $t\geq \frac12$, and $\theta(t)=4\eta^{\frac12}$ if $t\in(0,\frac12)$.
 In particular,
$$
\tilde{r}\geq (1-2^7\tau^{-1}\eta^{\frac14})r\mbox{ \ and \ }\tilde{s}\geq (1-2^7\tau^{-1}\eta^{\frac14})s\qquad \text{for }r,s\geq \eta^{\frac14},
$$
thus setting $t=r^{1-\lambda}s^\lambda$, we have
$$
\tilde{r}^{1-\lambda}\tilde{s}^\lambda\geq (1-2^7\tau^{-1}\eta^{\frac14})t\geq t-2^8\tau^{-1}\eta^{\frac14}.
$$
Therefore, if we define
\begin{equation}
\label{alphadef:control-sup}
\alpha=2^8\tau^{-1}\eta^{\frac14},
\end{equation}
then, for any $r\in(\alpha,\|\tilde{f}\|_\infty)$ and $s\in(\alpha,\|\tilde{g}\|_\infty)$, we deduce from
\eqref{rsinSrs0:control-sup} that $t=r^{1-\lambda}s^\lambda$ satisfies
\begin{eqnarray}
\nonumber
(1-\lambda)\widetilde{A}_r + \lambda\, \widetilde{B}_s&\subset&
(1-\lambda)\widetilde{A}_{\tilde{r}} + \lambda\, \widetilde{B}_{\tilde{s}}\subset
(1+2\eta^{\frac14})\left\{h>\tilde{r}^{1-\lambda}\tilde{s}^\lambda\right\}\\
\label{levelset-shift0:control-sup}
&\subset &
(1+2\eta^{\frac14})\{h>t-\alpha\}.
\end{eqnarray}
Next we replace $\tilde{f}$ by $\tilde{f}_0=\tilde{f}\chi_{\{\tilde{f}>\alpha\}}$ and $\tilde{g}$ by 
$\tilde{g}_0=\tilde{g}\chi_{\{\tilde{g}>\alpha\}}$. Then
 Lemma~\ref{logconvdim1}, \eqref{tildefg-close-stable}, and $\frac12<\|\tilde{f}\|_\infty,\|\tilde{g}\|_\infty<2$ (cp. \eqref{logconvmaxclose}), yield 
\begin{eqnarray}
\label{tildeff0int:control-sup}
\|\tilde{f}-\tilde{f}_0\|_1+\|\tilde{g}-\tilde{g}_0\|_1&\leq& 32\alpha\\
\label{tildef0g0:control-sup}
\H^1\big({\rm supp}\,\tilde{f}_0\big)+\H^1\big({\rm supp}\,\tilde{g}_0\big) &\leq& 32|\log \alpha|.
\end{eqnarray}
In particular, we deduce from \eqref{tildeff0int:control-sup} that
\begin{equation}
\label{0tilde0int:control-sup}
\|f-\tilde{f}_0\|_1+\|g-\tilde{g}_0\|_1\leq 2^6\alpha,
\end{equation}
hence
\begin{equation}
\label{tildefgintlow:control-sup}
\int_\R\tilde{f}_0,\;\int_\R\tilde{g}_0\geq 1-2^6\cdot\alpha,
\end{equation}
Consider now the log-concave function $\tilde{h}$ defined as
$$
\tilde{h}(z)=\sup_{z=(1-\lambda)x+\lambda\,y}\tilde{f}_0(x)^{1-\lambda}\tilde{g}_0(y)^{\lambda},
$$
which satisfies $\tilde{h}(z)\geq \alpha$ for any $z\in {\rm int}\,{\rm supp}\,\tilde{h}$ and
\begin{equation}
\label{tildehsup:control-sup}
\H^1\left({\rm supp}\,\tilde{h}\right) \leq 32|\log \alpha|
\end{equation}
(see \eqref{tildef0g0:control-sup}).
According to \eqref{tildefgintlow:control-sup} and the Pr\'ekopa-Leindler inequality, we have
\begin{equation}
\label{tildehintlow:control-sup}
\int_\R\tilde{h}\geq 1-2^6\alpha.
\end{equation}
It follows from the the definition of $\tilde{h}$ and \eqref{levelset-shift0:control-sup} that, for any $t>\alpha$, we have
\begin{equation}
\label{levelset-shift:control-sup}
\{\tilde{h}>t\}=\bigcup_{t=r^{1-\lambda}s^{\lambda}}
\left((1-\lambda)\widetilde{A}_r + \lambda\, \widetilde{B}_s\right)
\subset (1+2\eta^{\frac14})\{h>t-\alpha\}.
\end{equation}
To relate $\tilde{h}$ to $f$ and $g$, we deduce from \eqref{tildefgintlow:control-sup}
and \eqref{levelset-shift:control-sup} that
\begin{eqnarray}
\nonumber
\int_\R\tilde{h}&=&\int_\alpha^\infty\H^1\left(\{\tilde{h}>t\}\right)\,dt\leq (1+2\eta^{\frac14})\int_\alpha^\infty
\H^1\left(\{h>t-\alpha\}\right)\,dt
=(1+2\eta^{\frac14})\int_\R h\\
\label{tildehintupp:control-sup}
&<&1+4\eta^{\frac14}
\leq (1+2^9\alpha)\left(\int_\R\tilde{f}_0\right)^{1-\lambda}\left(\int_\R\tilde{g}_0\right)^{\lambda}.
\end{eqnarray}
Recalling that $\alpha=2^8\tau^{-1}\eta^{\frac14}$, thanks to Theorem~\ref{PLstab-logconv-dim1} 
 there exists $w \in \R$ such that
\[
 \int_{\R^n} |a_0^{\lambda}\tilde{f}_0-   \tilde{h}(\cdot +\lambda\,w)| + 
\int_{\R^n} |a_0^{\lambda-1}\tilde{g}_0- \tilde{h}(\cdot + (\lambda-1)w)| \lesssim \tau^{-\frac{2}3}
\eta^{\frac1{12}}|\log \alpha|^{\frac43}\int_{\R^n}\tilde{h}
\]
where  $a_0=\int_{\R^n} \tilde{g}_0/\int_{\R^n} \tilde{f}_0$.
 Also, by \eqref{tildefgintlow:control-sup} and the conditions $\int_\R\tilde{f},\,\int_\R\tilde{g}\leq 1+\eta$, it holds
 $$
1-2^{14}\tau^{-1}\eta^{\frac14}\leq \int_\R\tilde{f}_0,\;\int_\R\tilde{g}_0\leq 1+\eta,
$$
In particular $|a_0-1|\lesssim \tau^{-1}\eta^{\frac14}$, therefore
\[
 \int_{\R^n} |\tilde{f}_0-   \tilde{h}(\cdot +\lambda\,w)| + 
\int_{\R^n} |\tilde{g}_0- \tilde{h}(\cdot + (\lambda-1)w)| \lesssim \tau^{-\frac{2}3}
\eta^{\frac1{12}}|\log \alpha|^{\frac43}\int_{\R^n}\tilde{h}.
\]
Recalling \eqref{0tilde0int:control-sup}, this proves the first bound in the statement of
Proposition~\ref{thm:close-stable}.

To relate $\tilde{h}$ to $h$, consider the
auxiliary function 
$$
\tilde{h}_0(x)=
\left\{
\begin{array}{ll}
\tilde{h}((1+2\eta^{\frac14})x)-\alpha &\mbox{ if $x\in{\rm int}\,{\rm supp}\,\tilde{h}$}, \\
0& \mbox{ otherwise,}
\end{array}
\right.
$$
so that, if $t>\alpha$, then
\begin{equation}
\label{tildeh0level:control-sup}
\{\tilde{h}>t\}=(1+2\eta^{\frac14})\{\tilde{h}_0>t-\alpha\}.
\end{equation}
Comparing \eqref{tildeh0level:control-sup} and \eqref{levelset-shift:control-sup}, it follows that
$\tilde{h}_0\leq h$. In addition, \eqref{tildehintlow:control-sup} implies that
$$
1-2^7\alpha<(1+2\eta^{\frac14})^{-1}\int_\R \tilde{h}=\int_\R \tilde{h}_0\leq \int_\R h<1+\varepsilon,
$$
therefore 
\begin{equation}
\label{htildeh0:control-sup}
\|h-\tilde{h}_0\|_1<2^8\alpha.
\end{equation}
Next we claim that
\begin{equation}
\label{tildehtildeh0error:control-sup}
\tilde{h}((1+2\eta^{\frac14})x)< \tilde{h}(x)+2^7\tau^{-2} \eta^{\frac14}\qquad \text{for any $x\in {\rm supp}\,\tilde{h}$.}
\end{equation}
We observe that $t_0=r_0^{1-\lambda}s_0^\lambda\geq 1-2^6\tau^{-\frac32} \eta^{\frac14}$ according to
\eqref{fgintest-close-stable},
\eqref{logconvmaxclose}, and \eqref{r0s0def:control-sup}. Since $\tilde{f}$ and $\tilde{g}$ were translated to ensure
 $\tilde{f}_0(0)\geq r_0$ and 
$\tilde{g}_0(0)\geq s_0$, we deduce that $\tilde{h}(0)\geq t_0$. Using that $\tilde{h}$ is log-concave, we deduce that
that if $\tilde{h}(x)\leq t_0$, then $\tilde{h}((1+2\eta^{\frac14})x)\leq \tilde{h}(x)$ and the follows. 
On the other hand, if $\tilde{h}(x)> t_0$ then \eqref{tildehtildeh0error:control-sup} follows from $\|\tilde{h}\|_\infty\leq 1+32\tau^{-\frac32} \eta^{\frac12}$ (see \eqref{fgintest-close-stable} and
\eqref{logconvmaxclose})
and the bound $t_0\geq 1-2^6\tau^{-\frac32} \eta^{\frac14}$.

Thanks to \eqref{tildehtildeh0error:control-sup}, since
$\alpha\leq 2^7\tau^{-2} \eta^{\frac14}$ we get
\begin{eqnarray*}
\|\tilde{h}-\tilde{h}_0\|_1&=&\int_{{\rm supp}\,\tilde{h}}\left|\tilde{h}(x)-\tilde{h}((1+2\eta^{\frac14})x)+\alpha\right|\,dx\\
&=&\int_{{\rm supp}\,\tilde{h}}\left|\tilde{h}(x)+2^7\tau^{-2} \eta^{\frac14}-\tilde{h}((1+2\eta^{\frac14})x)+
(\alpha-2^7\tau^{-2} \eta^{\frac14})\right|\,dx\\
&\leq & \int_{{\rm supp}\,\tilde{h}}\tilde{h}(x)+2^7\tau^{-2} \eta^{\frac14}-\tilde{h}((1+2\eta^{\frac14})x)\,dx
+\int_{{\rm supp}\,\tilde{h}}2^7\tau^{-2} \eta^{\frac14}\,dx\\
&=&\left(1-\frac1{1+2\eta^{\frac14}}\right)\int_{{\rm supp}\,\tilde{h}}\tilde{h}(x)\,dx+
2\cdot \H^1({\rm supp}\,\tilde{h})\cdot 2^7\tau^{-2} \eta^{\frac14}.
\end{eqnarray*}
Since $\int_\R\tilde{h}<2$  and
 $\H^1({\rm supp}\,\tilde{h})\leq 32|\log \alpha|$ (see \eqref{tildehintupp:control-sup} and \eqref{tildehsup:control-sup}), we conclude that
$\|\tilde{h}-\tilde{h}_0\|_1<2^{14}\tau^{-2} \eta^{\frac14}|\log \alpha|$.
Combining this estimate with \eqref{htildeh0:control-sup} implies that
$\|h-\tilde{h}\|_1<2^{15}\tau^{-2} \eta^{\frac14}|\log \alpha|$. As $\alpha=2^8\tau^{-1}\eta^{\frac14}$, we have $|\log \alpha| \lesssim \max\{ |\log \tau|, |\log \eps|\} \lesssim |\log \eps|.$ Plugging this into the statements above, we obtain the original claim, which finishes the proof.  
\end{proof}

\section{The case of symmetric-rearranged functions}\label{sec:rearranged-version} For this part and for the remainder of the paper, we assume all the reductions and results from \S 2 to hold. 

As noticed in the beginning of the previous section, the \emph{symmetric decreasing rearrangements} of functions $f,g,h$ satisfying \eqref{eq:condition} and \eqref{eq:almost-eq}, denoted by $f^*,g^*,h^*$, also satisfy \eqref{eq:condition} and \eqref{eq:almost-eq} with the same constant, as rearrangements preserve $L^p-$norms. By changing these functions on a zero-measure
set, we may suppose that their level sets are all open. The main result of this section lays out the foundation for the analysis in the following ones, and can be summarized as follows:

\begin{theorem}\label{thm:rearranged-1-dim} There is an absolute constant $c > 0$ such that the following holds. Suppose $f,g,h : \R \to \R_{\ge 0}$ satisfy \eqref{eq:condition} and \eqref{eq:almost-eq} for $0 < \eps < c e^{-\frac{1000 |\log \tau|^4}{\tau^4}}.$ Then there exist even log-concave functions $\tilde{f}, \tilde{g}$ such that 
\[
\| f^* - \tilde{f}\|_1 + \|g^* - \tilde{g}\|_1 \lesssim \tau^{-\omega} \eps^{\frac{\tau}{2^{21} |\log \tau|}},
\]
where $\omega$ is an absolute constant given by $\omega = 6 + \frac{3\omega_0}{2},$ with $\omega_0$ as in Lemma~\ref{thm:near-concave-lemma}.
\end{theorem}

Here and henceforth, given a family of sets $\{S_\alpha\}$, we shall use the notation $\bigcup^*_{\alpha} S_{\alpha}$ to denote the union $\bigcup_{\alpha \colon S_{\alpha} \neq \emptyset} S_{\alpha}.$

\begin{proof}[Proof of Theorem~\ref{thm:rearranged-1-dim}] First, we may suppose without loss of generality that $\|f\|_{1} = \|g\|_{1} = 1,$ and that $\min\{\|f\|_{\infty}, \|g\|_{\infty}\} = \|f\|_{\infty} = 1.$ These assumptions, together with Lemma \ref{thm:control-sup}, imply that 
	$$
	0 \le \|g\|_{\infty} -1 \le 4 \tau^{-\frac32} \eps^{\frac12}.	
	$$ 
Consider, thus, the functions $a, b,c:\R \to \R_+$ defined so to satisfy,
for any $R\in \R$, 
\begin{equation*}
\begin{split} 
\{f^* > e^R\} &= (-a(R),a(R))=:\AC_R,\cr
\{g^* > e^R\} &= (-b(R),b(R)) =:\BC_R,\cr
\{h^* > e^R\} &= (-c(R),c(R)) =: \CC_R. \cr
\end{split}
\end{equation*}
By \eqref{eq:condition} applied to $h^*,$ we have
\begin{equation}\label{eq:pl-cond-sets}
\CC_T \supseteq \bigcup^{*}_{(1-\lambda )R+ \lambda S = T } \left\{(1- \lambda) \AC_R + \lambda \BC_S \right\}.
\end{equation}
Thus, as $\int f^* = \int g^* = 1,$ by a change of variables, we have 
\[
\eps \ge \int_{-\infty}^{\infty} \left(\H^1(\CC_{T}) - ((1-\lambda )\H^1(\AC_T) +  \lambda \H^1(\BC_T)) \right) e^T \, d T. 
\]
Notice that the map $T \mapsto \H^1(\CC_{T}) - (1-\lambda) \H^1(\AC_T)-\lambda \H^1(\BC_T)$ is, by \eqref{eq:pl-cond-sets} and the Brunn--Minkowski inequality, nonnegative for all $T>0$ for which $\AC_T,\BC_T \neq \emptyset.$ Reversing 
$$\AC_T = A_{e^T},\, \BC_T= B_{e^T},\, \CC_T = C_{e^T},$$
changing variables $e^T = t,$ we claim that we may find a set $F \subset \R_{+}$ such that $\H^1({\R_+}\setminus F) \lesssim \eps^{\frac14}$ and
\begin{equation}\label{eq:bound-level-sets}
 \left|\H^1(C_t) - (1-\lambda) \H^1(A_t)-\lambda\H^1(B_t)\right| \le \tau^{-\frac32} \eps^{\frac14}\qquad \forall\, t \in F.
\end{equation}
Indeed, let $S_1 = \{ t \ge 0; \H^1(C_t) \ge (1-\lambda) \H^1(A_t) + \lambda \H^1(B_t)\}.$ By the reductions made, we know that $S_1 \supseteq [0,1].$ If we denote by $S_2 = \R\setminus S_1,$ then 
\[
|\H^1(C_t) - (1-\lambda) \H^1(A_t) - \lambda \H^1(B_t) | \le \max\{\H^1(A_t),\H^1(B_t)\} \lesssim 1 \, \forall \, t \in S_1. 
\]
Thus, 
\[
\int_{S_1} |\H^1(C_t) - (1-\lambda) \H^1(A_t) - \lambda \H^1(B_t) | \, dt \lesssim \int_1^{1+4\tau^{-\frac32} \eps^{\frac12}} 1 \, dt \lesssim \tau^{-\frac32} \eps^{\frac12}. 
\]
By the fact that the integral $\int_{\R_{+}} (\H^1(C_t) - (1-\lambda) \H^1(A_t) - \lambda \H^1(B_t) ) \, dt \le \eps,$ we obtain that 
\[
\int_{\R_+} |\H^1(C_t) - (1-\lambda) \H^1(A_t) - \lambda \H^1(B_t) | \, dt \lesssim \tau^{-\frac32} \eps^{\frac12}.  
\]
By using Chebyshev's inequality, we obtain that the set of $t \ge 0$ where the integrand is larger than $ c \tau^{-\frac32} \eps^{\frac14}$  has measure at most $ c \cdot \eps^{\frac14},$ which finishes the proof of the claim. In particular, if $\AC_R, \BC_S \neq \emptyset, (1-\lambda) R+\lambda S=T,$ and $e^T = t \in F,$ we have 
\begin{equation}\label{eq:three-point-1}
(1-\lambda)a(R) + \lambda b(S) \le \left((1-\lambda)a + \lambda b\right)(T) + \tau^{-\frac32}\eps^{\frac14}. 
\end{equation}
Fix thus $M = \theta \log(1/\eps),$ with $\theta > 0$ small to be chosen later. Denote by $F_{M} = F \cap [e^{-M},e^M].$ With this definition, we have that the set 
$$\log(F_M) = \{T \in \R \colon e^T \in F_M\}$$ 
has large measure within $[-M,M].$ Indeed, recalling that $\H^1({\R_+}\setminus F) \le \eps^{\frac14}$, 
\begin{equation*}
\int_{\R} \chi_{[-M,M] \setminus\log( F_M)}(T) \, d T  \le e^M \int_{\R} \chi_{[-M,M] \setminus \log(F_M)}(T) \, e^T d T = \eps^{-\theta} \H^1([e^{-M},e^M]\setminus F) \le  \eps^{\frac14-\theta}. 
\end{equation*}
Thus, if $\theta < 1/8,$ $\H^1([-M,M]\setminus \log(F_M)) \le \eps^{\frac18}.$ 

Therefore, if $T_1,T_2 \in \log(F_M),$ and additionally 
$$T_{1,2} = \frac{1}{2-\lambda} T_1  + \frac{1-\lambda}{2-\lambda} T_2 \in \log(F_M), \qquad T_{2,1} = \frac{1}{2-\lambda} T_2 + \frac{1-\lambda}{2-\lambda} T_1 \in \log(F_M),$$
then the reduction in \cite[Remark~4.1]{FigalliJerison} shows that the following \emph{four-point inequalities} hold:
\begin{equation}\begin{split}\label{eq:four-point-1}
a(T_1) + a(T_2) &\le a(T_{1,2}) + a(T_{2,1}) + \frac{2}{\lambda} \tau^{-\frac32} \eps^{\frac14}, \cr 
b(T_1) + b(T_2) &\le b(T_{1,2}) + b(T_{2,1}) + \frac{2}{\lambda} \tau^{-\frac32} \eps^{\frac14}. \cr 
\end{split}\end{equation}
Inspired by this, we recall the statement of Lemma 3.6 in \cite{FigalliJerison} in the one-dimensional case: 

\begin{lemma}[Lemma 3.6 in \cite{FigalliJerison}]\label{thm:near-concave-lemma} Let $G \subset \R$ be a measurable subset and $\psi:G \to \R$ be a function, such that the following properties hold: 
\begin{enumerate}
 \item The four-point inequality 
 \begin{equation}\label{eq:four-point-inequality-2}
 \psi(T_1) + \psi(T_2) \le \psi(T_{1,2}) + \psi(T_{2,1}) + \sigma 
 \end{equation}
 holds, whenever $T_1,T_2,T_{1,2},T_{2,1} \in G$; 
 \item The convex hull $\text{co}(G) = \Omega$ satisfies $\H^1(\Omega\setminus G) \le \zeta;$ 
 \item There is $r \in (1/2,2)$ with $[-r,r] = \Omega;$
 \item The inequalities $-\kappa \le \psi(T) \le \kappa$ hold for all $T \in F,$ for some $\kappa \ge 1;$ 
 \item There is $H \subset \R$ such that 
 \begin{equation}\label{eq:small-measure-hull}
 \int_{H} \H^1(\text{co}(\{\psi > s\}) \setminus \{\psi > s\}) \, d s + \int_{\R \setminus H} \H^1(\{\psi > s\}) \le \zeta.
 \end{equation}
\end{enumerate}
Then there exist a concave function $\Psi : \Omega \to [-2\kappa,2\kappa]$, and an absolute constant $c>0$, such that
\begin{equation}\label{eq:near-concave}
\int_{G} |\Psi(T) - \psi(T)| \, d T \le c \kappa \tau^{-\omega_0} (\sigma + \zeta)^{\alpha_{\tau}},
\end{equation}
where we let $\alpha_{\tau} = \frac{\tau}{16 |\log \tau|},$ and $\omega_0 > 0$ is an absolute constant. 
\end{lemma}

We are almost ready to apply Lemma~\ref{thm:near-concave-lemma}: we change variables and set $\tilde{a}(T') = a(M T').$ 

If $T'_1,T'_2,T'_{1,2},T'_{2,1} \in \log(F_M)/M$ and $\lambda \in [\tau, 1-\tau],$ then the four-point inequality 
\eqref{eq:four-point-inequality-2} holds for $\tilde{a},$ with $\sigma = \frac{2\eps^{\frac14}}{\tau^{5/2}}.$ Moreover, from the properties of $\log(F_M),$ we obtain 
$$\H^1(\text{co}(\log(F_M)/M) \setminus (\log(F_M)/M)) \le \eps^{\frac18}.$$

From that, we see that $\tilde{\Omega}_M := \text{co}(\log(F_M)/M)$ is an interval that differs by at most $\eps^{\frac18}$ from the interval $[-1,1],$ and thus can be written as $T_0 + I,$ with $I=[-r,r]$ and $|r-1| \le 2\eps^{\frac18}$,
and $T_0 \in \R$ with $|T_0| \le \eps^{\frac18}.$ 

Defining the function $\tilde{a}'(T'') = \tilde{a}(T' + T_0)$ preserves conditions (1), (2), (4), and (5), in Lemma~\ref{thm:near-concave-lemma}. In addition, now also condition (3)  is fulfilled. Furthermore, by Lemma~\ref{thm:cutting-support}, we have $\tilde{a}'$ is bounded in absolute value by $\kappa = \frac{c}{\tau^4} |\log\eps|^{\frac4\tau},$ with $c$ an absolute constant. 

Finally, as the function $a$ is nonincreasing on 
$\R,$ the level sets of $\tilde{a}'$ are all intervals.
Hence we may take $H$ to be the support of $\tilde{a}'$ in \eqref{eq:near-concave}, and $\zeta = 4\eps^{\frac18}$. 

Therefore, by Lemma~\ref{thm:near-concave-lemma}, there is a concave function $\tilde{\mathfrak{a}}' :\tilde{\Omega}'_M := \tilde{\Omega}_M - T_0 \to [-2\kappa, 2\kappa]$ such that 
\[
\int_{\log(F_M)/M - T_0} |\tilde{\mathfrak{a}}'(T) - \tilde{a}'(T)| \, d T \le \kappa \tau^{-\omega_0} \cdot \frac{\eps^{\frac{\alpha_{\tau}}8}}{\tau^{5\alpha_{\tau}/2}}.
\]
Thus, the function $\tilde{\mathfrak{a}}(T) = \tilde{\mathfrak{a}}'(T - T_0)$ satisfies 
\[
\int_{\log(F_M)/M} |\tilde{\mathfrak{a}}(T) - \tilde{a}(T)| \, d T \lesssim  |\log\eps|^{\frac4\tau} \frac{\eps^{\frac{\alpha_{\tau}}8}}{\tau^{4+\omega_0}}.
\]
This follows from the definition of $\kappa$ and the fact that $\tau^{\alpha_{\tau}} = e^{-\tau/16},$ which is bounded from below and above whenever $\tau \in (0,1/2].$ Changing variables $T = T'/M$ above yields that $\mathfrak{a}(T) = \tilde{\mathfrak{a}}(T/M)$ satisfies  (recall that $M = \theta \log(1/\eps)$)
\begin{equation}\label{eq:almost-concave-a}
\int_{\log(F_M)} |\mathfrak{a}(T') - a(T')| \, d T' \lesssim |\log\eps|^{1+\frac4\tau} \frac{\eps^{\frac{\alpha_{\tau}}8}}{\tau^{4+\omega_0}}.
\end{equation}
We observe that, if we denote by $\Omega_M = M \tilde{\Omega}_M$ the domain of definition of $\mathfrak{a}$,
then it follows from the considerations above that $\H^1([-M,M] \setminus \Omega_M) \lesssim  |\log\eps| \eps^{\frac18}.$ 

Notice that the process above can be adapted verbatim to $b,$ and we find a concave function $\mathfrak{b}:\Omega_M \to [-2\kappa, 2\kappa]$ such that 
\begin{equation}\label{eq:almost-concave-b} 
\int_{\log(F_M)} |\mathfrak{b}(T') - b(T')| \, d T' \lesssim  |\log\eps|^{1+\frac4\tau} \frac{\eps^{\frac{\alpha_{\tau}}8}}{\tau^{4+\omega_0}}.
\end{equation}

Let, for shortness, $\omega_1 := 4+\omega_0.$ We must now ensure that $\mathfrak{a}, \mathfrak{b}$ satisfy the requirements of distribution functions. Indeed, in case $\mathfrak{a}$, $\mathfrak{b}$ are both nonincreasing on the subinterval $I_M = [-3M/4, 3M/4] \subset \Omega_M,$ we do not change them. 

On the other hand, if either $\mathfrak{a}$ or $\mathfrak{b}$ are not nonincresing on such a large interval, we use Chebyshev's inequality in conjunction with \eqref{eq:almost-concave-a} and \eqref{eq:almost-concave-b}. 

This implies that there is a set $\mathcal{F} \subset \log(F_M)$ such that $\H^1(\log(F_M) \setminus \mathcal{F}) \le \tau^{-\frac{\omega_1}{2}}\eps^{\frac{\alpha_{\tau}}{32}},,$ and 
\[
|\mathfrak{b}(T) - b(T)| + |\mathfrak{a}(T)-a(T)| \lesssim \tau^{-\frac{\omega_1}{2}}\eps^{\frac{\alpha_{\tau}}{32}}, \qquad \forall\, T \in \mathcal{F}.
\]
Changing $\mathfrak{a},\mathfrak{b}$ on a zero measure set, we may suppose that both are lower semicontinuous. Suppose then without loss of generality that $\mathfrak{a}$ attains its maximum at a point $T_0 \in I_M.$ 

As $\H^1(\Omega_M \setminus \mathcal{F}) \lesssim \tau^{-\frac{\omega_1}{2}}\eps^{\frac{\alpha_{\tau}}{32}},,$ there is a point $T_1 \in \mathcal{F}$ such that 
$$|T_0 - T_1| \lesssim \tau^{-\frac{\omega_1}{2}}\eps^{\frac{\alpha_{\tau}}{32}},.$$ 
Analogously, there is a point $T_2 \in \mathcal{F}$ such that $|T_2 + M| \lesssim \tau^{-\frac{\omega_1}{2}}\eps^{\frac{\alpha_{\tau}}{32}},,$ thus,
\begin{equation}\begin{split}\label{eq:tightness-concave}
\mathfrak{a}(T_0) - \mathfrak{a}(T_2) &\le |\mathfrak{a}(T_2) - a(T_2)| + a(T_1) - a(T_2) + |a(T_1) - \mathfrak{a}(T_1)| + |\mathfrak{a}(T_1) - \mathfrak{a}(T_0)| \cr
				      &\le  c \tau^{-\frac{\omega_1}{2}}\eps^{\frac{\alpha_{\tau}}{32}},+ |\mathfrak{a}(T_1) - \mathfrak{a}(T_0)|. \cr
\end{split}\end{equation}
On the other hand, by concavity,
\begin{equation}\label{eq:concavity-use-1}
\mathfrak{a}(T_1) \ge \gamma \mathfrak{a}(T_0) + (1-\gamma) \mathfrak{a}(T_2), \qquad \text{with $\gamma \in (0,1)$ such that }\gamma T_0 + (1-\gamma)T_2 = T_1.
\end{equation}
By the way we chose $T_0,T_1,T_2,$ we get 
\[
\tau^{-\frac{\omega_1}{2}}\eps^{\frac{\alpha_{\tau}}{32}}, \gtrsim |T_1 - T_0| = (1-\gamma)|T_0 - T_2| \ge \left(\frac{M}{4} - c \tau^{-\frac{\omega_1}{2}}\eps^{\frac{\alpha_{\tau}}{32}}, \right)(1-\gamma).
\]
Thus, if $\eps>0$ is sufficiently small, we have 
\[
\gamma \ge 1 - 10 \tau^{-\frac{\omega_1}{2}}\eps^{\frac{\alpha_{\tau}}{64}},.
\]
Also, by boundedness of $\mathfrak{a},$ 
\begin{equation}\label{eq:bound-difference-a}
|\mathfrak{a}(T_1) - \mathfrak{a}(T_0)| \lesssim |\log\eps|^{\frac4\tau} \tau^{-\frac{3\omega_1}{2}}\eps^{\frac{\alpha_{\tau}}{64}},.
\end{equation}
Combining \eqref{eq:bound-difference-a} and \eqref{eq:tightness-concave} implies 
$$\mathfrak{a}(T_0) \le \mathfrak{a}(T_2) + c |\log\eps|^{\frac4\tau} \tau^{-\frac{3\omega_1}{2}}\eps^{\frac{\alpha_{\tau}}{64}},,$$
where $c > 0$ is an absolute constant, and so, by monotonicity,
\begin{equation}\label{eq:almost-constant-concave}
\mathfrak{a}(T_0) \le \mathfrak{a}(T) + c |\log\eps|^{\frac4\tau} \tau^{-\frac{3\omega_1}{2}}\eps^{\frac{\alpha_{\tau}}{64}},, \qquad \forall \, T \in I_M, \,T < T_0.
\end{equation}
We thus define 
\[
\tilde{\mathfrak{a}}(T) = \begin{cases}
                          \mathfrak{a}(T),\, & \text{ if } T \in I_M, T \ge T_0; \cr 
                          \mathfrak{a}(T_0),\, & \text{ if } T \in I_M, T < T_0. \cr
                          \end{cases} 
\]
This new function, besides being concave, is also nonincreasing on $I_M,$ and, by \eqref{eq:almost-concave-a} and \eqref{eq:almost-constant-concave},
\[
\int_{\log(F_M) \cap I_M} |\tilde{\mathfrak{a}}(T) - a(T)| \, d T \lesssim  |\log\eps|^{1+\frac4\tau}  \tau^{-\frac{3\omega_1}{2}}\eps^{\frac{\alpha_{\tau}}{64}},.
\]
As both $a, \tilde{\mathfrak{a}}$ are bounded by $c|\log\eps|^{\frac4\tau}/\tau^4 $ on $I_M$ and $\H^1(I_M \setminus \log(F_M)) \le \eps^{\frac18},$ we conclude moreover that 
\[
\int_{I_M} |\tilde{\mathfrak{a}}(T) - a(T)| \, d T \lesssim |\log\eps|^{1+\frac4\tau} \tau^{-\frac{3\omega_1}{2}}\eps^{\frac{\alpha_{\tau}}{64}},.
\]
By symmetry, the same method can be applied to the function $b.$ Given the two resulting concave functions $\tilde{\mathfrak{a}}, \tilde{\mathfrak{b}},$ they define an almost-everywhere
unique pair $\tilde{f},\tilde{g}$ of functions such that 
\begin{equation*}
\{x \in \R \colon \tilde{f}(x) > t\} = (-\tilde{\mathfrak{a}}(\log t), \tilde{\mathfrak{a}}(\log t)),\qquad
\{x \in \R \colon \tilde{g}(x) > t\} = (-\tilde{\mathfrak{b}}(\log t), \tilde{\mathfrak{b}}(\log t)),
\end{equation*}
whenever $\log t \in I_M$ (that is, $t \in (\eps^{\frac{3\theta}4},\eps^{-\frac{3\theta}4})$), 
\begin{equation*}
\text{supp}(\tilde{f}) = \bigcup_{t \in (\eps^{\frac{3\theta}4},\eps^{-\frac{3\theta}4})} (-\tilde{\mathfrak{a}}(\log t), \tilde{\mathfrak{a}}(\log t)), \qquad
\text{supp}(\tilde{g}) = \bigcup_{t \in (\eps^{\frac{3\theta}4},\eps^{-\frac{3\theta}4})} (-\tilde{\mathfrak{b}}(\log t), \tilde{\mathfrak{b}}(\log t)), 
\end{equation*}
and
$\{x \in \R \colon \tilde{f}(x) > t\} = \{x \in \R \colon \tilde{g}(x) > s\}  = \emptyset$
for $t,s > \eps^{-\frac{3\theta}4}$ or whenever $\tilde{\mathfrak{a}}(\log t) = 0 = \tilde{\mathfrak{b}}(\log s).$

We claim that these functions
are log-concave. 
Indeed, if $\tilde{f}(x_1) > s_1$ and $\tilde{f}(x_2) > s_2$ with $s_1,s_2 \in (\eps^{\frac{3\theta}4},\eps^{-\frac{3\theta}4})$ then 
$$x_1 \in (-\tilde{\mathfrak{a}}(\log s_1), \tilde{\mathfrak{a}}(\log s_1)), \,\,x_2 \in (-\tilde{\mathfrak{a}}(\log s_2), \tilde{\mathfrak{a}}(\log s_2)).$$
By concavity, for any $t \in (0,1),$ 
\begin{equation*}\begin{split}
tx_1 + (1-t)x_2 &\in (-t\tilde{\mathfrak{a}}(\log s_1)-(1-t)\tilde{\mathfrak{a}}(\log s_2), t\tilde{\mathfrak{a}}(\log s_1) + (1-t)\tilde{\mathfrak{a}}(\log s_2))  \cr
                &\subseteq (-\tilde{\mathfrak{a}}(\log(s_1^t s_2^{1-t})),\tilde{\mathfrak{a}}(\log(s_1^t s_2^{1-t}))). \cr
\end{split}\end{equation*}
Thus $\tilde{f}(tx_1 + (1-t)x_2) > s_1^t s_2^{1-t},$ which concludes in this case.

The case $\max\{s_1,s_2\} > \eps^{-\frac{3\theta}4}$ or $\tilde{\mathfrak{a}}(\max\{\log s_1,\log s_2\}) = 0$ is trivial by definition.
Also, 
if $s_1 \in (0,\eps^{\frac{3\theta}4}),$ then $x_1 \in (-\tilde{\mathfrak{a}}(\log t_0),\tilde{\mathfrak{a}}(\log t_0)),$ for $t_0 \in (\eps^{\frac{3\theta}4},\eps^{-\frac{3\theta}4}),$ and thus we reduce to the previous one. By symmetry, the same holds for $\tilde{g},$ and the claim is proved. 

Finally, it remains to prove that $\|f-\tilde{f}\|_1 + \|g-\tilde{g}\|_1$ is small. By layer-cake representation, choosing $\theta = \alpha_{\tau}/100$ we have 
\begin{equation}\begin{split}\label{eq:layer-cake-1} 
\|f - \tilde{f}\|_1 & = \int_0^{\infty} \H^1(\{ f > t\} \Delta \{ \tilde{f} > t\}) \, d t = \int_{\R} |a(T) - \tilde{\mathfrak{a}}(T)| \, e^T \, d T \cr 
		    & \le \int_0^{\eps^{\frac{3\theta}4}} \left(\H^1(\{ f > t\}) + \H^1(\{\tilde{f} > t\})\right) \, d t + \eps^{-\frac{3\theta}4} \int_{I_M} |a(T) - \tilde{\mathfrak{a}}(T)| \, d T \cr 
		    & \lesssim \frac{\eps^{\frac{3\theta}4} |\log\eps|^{\frac4\tau}}{\tau^4} +  |\log\eps|^{1+\frac4\tau} \eps^{\frac{\alpha_\tau}{64}-\frac{3\theta}{4}}\tau^{-\frac{3\omega_1}{2}} \lesssim  \eps^{\frac{\alpha_{\tau}}{128}}|\log\eps|^{1+\frac4\tau} \tau^{-\frac{3\omega_1}{2}},
\end{split}\end{equation}
where we used  $\|f\|_{\infty},\|g\|_{\infty} \le 2$ and Lemma~\ref{thm:cutting-support}. Naturally, all such considerations hold in the exact same manner for $g, \tilde{g}.$

We now notice that, if $\eps >0 $ satisfies the smallness condition as in the statement of the result, then we may bound 
\[
|\log \eps| ^{1+\frac4\tau} \eps^{\frac{\alpha_{\tau}}{128}} \le \eps^{\frac{\alpha_{\tau}}{256}}. 
\] 
By Proposition~\ref{thm:close-stable}, this is enough to conclude the case of symmetrically decreasing functions. As we do not need an explicit estimate on the distance between $h$ and a log-concave function, we omit the final bound one could obtain using that proposition, limiting ourselves thus to the statement of Theorem~\ref{thm:rearranged-1-dim}  
\end{proof}

\section{The general case}\label{sec:general-case}

We now turn to the general case, assuming the results in the previous subsection. We shall prove the following result:

\begin{theorem}\label{thm:dim-1-stab} There is an explicitly computable constant $c_0 > 0$ such that the following holds. For $\tau \in (0,\frac{1}{2}]$ and $\lambda \in [\tau,1-\tau],$ if $f,g,h: \R \to \R_{\ge 0}$ are measurable functions for which \eqref{eq:condition} and \eqref{eq:almost-eq} hold, with $0 < \eps < c_0 e^{-M(\tau)},$ then there exist a log-concave function $\tilde{h}$ and $w \in \R$ such that 
\[
\int_{\R} |h-\tilde{h}| + \int_{\R} |a^{\lambda}f-   \tilde{h}(\cdot + \lambda\,w)| + 
\int_{\R} |a^{\lambda-1}g- \tilde{h}(\cdot + (\lambda-1)w)| <c_0\,
\frac{\varepsilon^{Q(\tau)}}{\tau^{\omega}} \int_{\R} h,
\]
where $\omega = \frac{5}{2} + \frac{\omega_0}{8},$ with $\omega_0$ being the exponent of $\tau$ in Lemma~\ref{thm:near-concave-lemma}, $M(\tau) =  10^{40} (\omega_0 + 4) \frac{|\log(\tau)|^4}{\tau^4} ,$ and $Q(\tau) = \frac{\tau^4}{2^{100} |\log \tau|^4}.$
\end{theorem}

As pointed out in the introduction, in order to prove such a result we shall break the proof into several steps.\\

\noindent\textbf{$\bullet$ Step 1: finding regular functions $\overline{f},\overline{g},\overline{h}$ that satisfy \eqref{eq:condition} and \eqref{eq:almost-eq} with a possibly smaller power of $\eps$.} Once more, we assume the reductions made in Sections~\ref{sec:tail-estimates} and~\ref{sec:rearranged-version} to hold. That is, we have $\|f\|_1 = \|g\|_1 = 1, \, \min\{\|f\|_{\infty}, \|g\|_{\infty}\} = \|f\|_{\infty} = 1.$ Lemma~\ref{thm:control-sup} yields then that 
$$\|g\|_{\infty} \in (1,1+c\tau^{-\frac32} \eps^{\frac12}).$$
Also, as $\|f\|_1 = \|g\|_1 = 1,$ using notation from Lemma~\ref{thm:cutting-support}, 
\[
\eps > \int_0^{\infty} \left( \H^1 (C_t) - (1-\lambda) \H^1(A_t) - \lambda \H^1(B_t) \right) \, d t \ge 0. 
\]
Thus, by an argument entirely analogous to that used to prove \eqref{eq:bound-level-sets}, we obtain once more 
\[
\tau^{-\frac32} \eps^{\frac12} \gtrsim \int_0^{\infty} \left| \H^1 (C_t) - (1-\lambda)\H^1(A_t) - \lambda \H^1(B_t) \right| \, d t.
\]
We obtain once more that the maximal measurable set $F \subset (0,+\infty)$ of $t \in \R_+$ so that 
\begin{equation}\label{eq:almost-brunn-minkowski-1d}
\left|\H^1(C_t) -(1-\lambda)\H^1(A_t) -\lambda(B_t) \right| \le \tau^{-\frac32} \eps^{\frac14}, \,\, \forall \, t \in F
\end{equation}
satisfies $\H^1(\R_+ \setminus F) \lesssim \eps^{\frac14}.$ Moreover, if $t < 1 - c \tau^{-\frac32} \eps^{\frac12},$ we know that $C_t \supset (1-\lambda)A_t + \lambda B_t.$ Thus, 
\begin{equation}\label{eq:near-eq-brunn-minkowski}
0 \le \H^1(C_t) - (1-\lambda)\H^1(A_t) - \lambda \H^1(B_t) \lesssim \tau^{-\frac32}  \eps^{\frac14}, \, \, \forall \, t \in F \cap (0,1-c\tau^{-\frac32} \eps^{\frac12}).
\end{equation}
We now wish to employ Freiman's theorem in order to conclude that the convex hull of the level sets $A_t,B_t$ are not to far off from $A_t, B_t$ themselves. 

In fact, by the considerations in Section~\ref{sec:rearranged-version}, we know that there are log-concave functions $\tilde{f}^*, \tilde{g}^*$such that 
\[
\|f^* - \tilde{f}^*\|_1 + \|g^* - \tilde{g}^*\|_1 \lesssim \tau^{-\frac{3\omega_1}{2}} \eps^{\frac{\alpha_{\tau}}{256}},
\]
where $f^*,g^*$ denote the symmetric decreasing rearrangements of $f,g$, respectively. By the reductions in the proof of Proposition~\ref{thm:close-stable},
we may suppose that \eqref{logconvmaxclose} holds for the functions $\tilde{f}^*,\tilde{g}^*.$ In particular, applying it in conjunction with Lemma~\ref{logconvdim1}  to these functions, we conclude that 
\[
\H^1(\{t > 0 \colon \H^1(\{\tilde{f}^* > t\}) \le \eps^{\delta} \}) \lesssim \eps^{\delta},
\]
for all $\delta>0.$ By writing 
\[
\|f^* - \tilde{f^*}\|_1 = \int_0^{\infty} \H^1(\{f^* > t\} \Delta \{\tilde{f^*} > t\}) \, d t \lesssim \tau^{-\frac{3\omega_1}{2}} \eps^{\frac{\alpha_{\tau}}{256}}
\]
and using the argument with Chebyshev's inequality we have extensively employed throughout this manuscript, we obtain 
\[
\H^1(\{t > 0 \colon \H^1(\{f^* > t\}) \le \eps^{\delta} \}) \lesssim \eps^{\delta}
\]
for all $\delta\in (0, \frac{\alpha_{\tau}}{1024}),$ and $\varepsilon >0$ sufficiently small (independently of $\tau>0$). Thus, by  equimeasurability of the rearrangement, 
\[
\H^1(\{t > 0 \colon \H^1(\{f > t\}) \le \eps^{\delta} \}) \lesssim \eps^{\delta},
\]
for all $\delta < \alpha_{\tau}/1024.$ In particular, we see that 
$$\H^1(A_t) > \eps^{\frac{\alpha_{\tau}}{2048}}, $$
whenever $t \in F' \subseteq F \cap (0,1-c\tau^{-\frac32}\eps^{\frac12}),$ where $\H^1(F\setminus F') \le \eps^{\frac{\alpha_{\tau}}{2048}}.$ The same holds for $g,$ and thus we denote still by $F'$ the set where the above properties hold for both $f$ and $g$.

Noticing that, for $\eps \leq  \tau^4 \ll 1$,
$$\min\{\H^1(A_t),\H^1(B_t)\} > \eps^{\frac{\alpha_{\tau}}{2048}} \gg \tau^{-\frac32} \eps^{\frac14}, \qquad \forall \, t \in F',$$
thanks to \eqref{eq:almost-brunn-minkowski-1d} we can apply Freiman's theorem. This yields that
\begin{equation}\label{eq:freiman-consequence}
\H^1(\text{co}(A_t) \setminus A_t) + \H^1(\text{co}(B_t) \setminus B_t) \lesssim \tau^{-\frac32} \eps^{\frac14},
\end{equation}
for all  $t \in F'.$ Notice also that, since the sets $\{A_t\}_{t > 0}$ are nested, the same property holds for their convex hulls $\{\text{co}(A_t)\}_{t>0}.$ 

With this in mind, we set 
\begin{equation}\label{eq:convex-hulls}
\text{co}(A_t) = (a_f^1(t),b_f^1(t)), \qquad \text{co}(B_t) = (a_g^1(t),b_g^1(t)). 
\end{equation}
The main idea is to  slightly change the functions $a_f^1,a_g^1,b_f^1,b_g^1$, in order to construct two functions $\overline{f},\overline{g}$ close to $f,g$ respectively, and whose level sets are intervals coinciding with $\text{co}(A_t), \text{co}(B_t)$ for the vast majority of levels $t > \eps^{\theta},$ where $\theta > 0$ will be a small constant to be chosen later. 

By redefining on a set of zero measure, we may assume that the functions $a_f^1,a_g^1,b_f^1,b_g^1$ are all right-continuous. Then we define
\begin{equation}\begin{split}
b_f(t) = \sup_{t'>t, t' \in F'} b_f^1(t'), & \qquad b_g(t) = \sup_{t' > t, t' \in F'} b_g^1(t'), \cr 
a_f(t) = \inf_{t'>t, t' \in F'} a_f^1(t'), & \qquad a_g(t) = \inf_{t' > t, t' \in F'} a_g^1(t'). \cr
\end{split}\end{equation}

The functions $a_f,a_g,b_f,b_g$ defined in such a way are all, by definition, monotone. Moreover, modifying on a zero-measure set, we may suppose them to be right-continuous as well. 

Let now $\theta > 0$ be a fixed parameter, whose exact value we shall determine later. We define 
$$(\overline a_f,\overline b_f) = (a_f(\eps^{\theta}),b_f(\eps^{\theta})).$$
As $\H^1((0,1-c\tau^{-\frac32} \eps^{\frac12}) \setminus F') \le \eps^{\frac{\alpha_{\tau}}{2048}},$ as long as we choose $\theta < \alpha_{\tau}/2^{12}$ we may always find a point $t_0 \in F'$ so that $\frac{1}{100}\eps^{\theta} <t_0 < \eps^{\theta}.$ Thus, for all $t \ge \eps^{\theta},$ \eqref{eq:freiman-consequence} yields 
\begin{equation}
(b_f(t) - a_f(t)) \le (b_f(t_0) - a_f(t_0)) \le \H^1(A_{t_0}) + c\tau^{-\frac32} \eps^{\frac14} \lesssim \tau^{-4} |\log\eps|^{\frac4\tau},
\end{equation}
where we used Lemma~\ref{thm:cutting-support} in the last inequality. We then build the function $\overline{f}$ (the construction of $\overline{g}$ being analogous) 
as 
\[
\overline{f}(x) = \begin{cases} 
		0, & \, \text{ if } x \not \in (\overline a_f,\overline b_f), \cr
                \sup \{t \colon a_f(t) < x \} & \, \text{ if } x \le a_f(1), \cr 
                1 & \, \text{ if } x \in (a_f(1),b_f(1)), \cr 
                \sup \{t \colon x < b_f(t) \} & \, \text{ if } x \ge b_f(1). \cr 
               \end{cases}
\]
Notice now that, for $s \in (0,1),$ 
\begin{equation*}\begin{split}
& \{x \in \R \colon \overline{f}(x) > s\} =\cr 
\{ x \in \R \colon \exists\, t > s \text{ so that either } & a_f(t) < x  \text{ and } x \le a_f(1) \text{ or } b_f(t) > x \ge b_f(1)\} \cup (a_f(1),b_f(1))  \cr 
= \bigcup_{t > s} (a_f(t),b_f(t)) & = \left(\inf_{t > s} a_f(t), \sup_{t> s} b_f(t)\right) = (a_f(s),b_f(s)). \cr 
\end{split}\end{equation*}
Notice that we used the hypothesis of right-continuity of $a_f,b_f$ in order to obtain the last equality above. Thus, we have 
\[
\overline{A}_t =: \{ \overline{f} > t \} = \text{co}(A_t), \qquad \forall \, t \in F'. 
\]
This allows us to estimate 
\begin{multline}\label{eq:distance-regular-function}
\int_{\R} |\overline{f}(x) - f(x)| \, d x = \int_0^{\infty} \H^1(A_t \Delta \overline{A}_t) \, d t \le \int_0^{\eps^{\theta}} \left( \H^1(A_t) + \H^1(\overline{A}_{t_0})\right) \, d t \\
+ \int_{(\eps^{\theta},1) \cap F'} \H^1(\text{co}(A_t) \setminus A_t) \, d t 
					      + \int_{(\eps^{\theta},1) \setminus F'} \left( \H^1(A_t) + \H^1(\overline{A}_t)\right) \, d t  \lesssim \tau^{-4} \eps^{\theta} |\log\eps|^{\frac4\tau}, \cr 
\end{multline}
where we used \eqref{eq:freiman-consequence}, $\theta <\alpha_{\tau}/2^{12}$, and once more Lemma~\ref{thm:cutting-support}. The same conclusion holds in an entirely analogous way for $\|g-\overline{g}\|_1$.

We now build a function $\overline{h}$ so that \eqref{eq:condition} and \eqref{eq:almost-eq} are satisfied. In fact, we take the most natural choice 
\[
\overline{h}(z) = \sup_{(1-\lambda)x + \lambda y = z} \overline{f}(x)^{1-\lambda} \overline{g}(y)^{\lambda}.
\]
The level sets $\overline{C}_t = \{x \in \R \colon \overline{h}(x) > t\}$ satisfy, by definition,
\[
\overline{C}_t = \bigcup_{r^{1-\lambda}s^{\lambda} = t}^* ((1-\lambda)\overline{A}_r + \lambda \overline{B}_s).
\]
As the level sets of $\overline{f},\overline{g}$ are intervals, the function $\overline{h}$ is measurable. It remains to verify that we have a control of the form 
\[
\int_{\R} \overline{h} \le 1 + c(\tau)\eps^{\gamma}, 
\]
for some $\gamma>0$ and some function $c(\tau)>0.$ The strategy here is similar to the proof of Proposition~\ref{thm:close-stable}. 

First, we may choose $\theta = \alpha_{\tau}/2^{13}$ in \eqref{eq:distance-regular-function}, so that we obtain 
\begin{equation}\label{eq:close-to-bubble}
\| f - \overline{f}\|_1 = \int_0^{\infty} \H^1(\{f > t\} \Delta \{\overline{f} > t\}) \, d t \lesssim \tau^{-4}\eps^{\frac{\alpha_{\tau}}{2^{13}}}|\log\eps|^{\frac4\tau},
\end{equation}
(with the same estimate holding for $g, \overline{g}$) and then use Chebyshev's inequality in order to conclude that 
\begin{equation}\label{eq:level-set-control-bubble}
\H^1(\{t > 0 \colon \H^1(\{\overline{f} > t\}) \le \eps^{\delta} \}) \lesssim \eps^{\delta},
\end{equation}
for all $\delta < \alpha_{\tau}/2^{15}.$ Then, we fix $\gamma_0 < \alpha_{\tau}/2^{15}$ and define $\overline{S} \subset (0,+\infty)$ to be the largest measurable subset of $(0,+\infty)$ satisfying:
\begin{enumerate}
 \item $\min\{\H^1(\{\overline{f} > t\}),\H^1(\{\overline{g} > t\})\} > \eps^{\gamma_0}$ for all $t \in \overline{S} \cap (0,1+c\tau^{-4}\eps^{\frac12})$);
 \item $\H^1(\{f > t\} \Delta \{\overline{f} > t\}) + \H^1(\{g > t\} \Delta \{ \overline{g} > t\}) \, \lesssim \eps^{\frac{\alpha_{\tau}}{2^{15}}}$ for all $t \in \overline{S}.$
\end{enumerate}
By \eqref{eq:close-to-bubble} and \eqref{eq:level-set-control-bubble}, we have $\H^1(\R_+ \setminus \overline{S}) \lesssim \tau^{-4}\eps^{\gamma_0}.$ Thus, for some absolute constant $c > 0$, 
there is an element $r_0 \in (1-c\tau^{-4}\eps^{\gamma_0},1+c\tau^{-4}\eps^{\gamma_0}) \cap \overline{S}.$ Fix this element until the end of the proof. 

Note that transformations of the form
$$(f,g,h) \mapsto (f(\cdot-x_0),g(\cdot+x_0),h),\qquad (f,g,h) \mapsto (f(\cdot-x_0),g(\cdot-x_0),h(\cdot-x_0))$$
preserve \eqref{eq:condition} and \eqref{eq:almost-eq} with the same constant. Also, they leave the set $\overline{S}$ defined above unaltered. Hence, with no loss of generality, we may suppose that the barycenters of $\{\overline{f} > r_0\}$ and $\{\overline{g} > r_0\}$ both coincide 
with the origin. Assume this additional fact until the end of the proof as well. 

Now we employ the same strategy as in the final part of the proof of Proposition~\ref{thm:close-stable}. Fix $t > \eps^{\frac{\tau\gamma_0}2}.$ It is not hard to see that the set $\{\overline{h} > t\}$ splits as 
\begin{multline*}
\overline{C}_t  = \bigcup^*_{\substack{{r^{1-\lambda}s^{\lambda}=t} \\{r,s \in \overline{S}} \\{r_0 > r,s > \eps^{\gamma_0}}}} ((1-\lambda)\overline{A}_r + \lambda\overline{B}_s) \cup \bigcup^*_{\substack{{r^{1-\lambda}s^{\lambda}=t} \\{r,s \in \overline{S}} \\{\text{ either } r> r_0 \text{ or } s > r_0 }}} ((1-\lambda)\overline{A}_r + \lambda\overline{B}_s) \\
\cup \bigcup^*_{\substack{{r^{1-\lambda}s^{\lambda}=t} \\{\text{ either } r \not \in \overline{S} \text{ or } s \not \in \overline{S}}}} ((1-\lambda)\overline{A}_r + \lambda\overline{B}_s) =: \overline{C}_t^1 \cup \overline{C}_t^2 \cup \overline{C}_t^3. 
\end{multline*}

\noindent\textit{Case 1: Analysis of $\overline{C}_t^1.$} By Young's convolution inequality and the definition of $\overline{S},$ we have 
\begin{equation}\begin{split}\label{eq:convolution-bound-2}
\| \chi_{(1-\lambda)A_r} * \chi_{\lambda B_s} - \chi_{(1-\lambda)\overline{A}_r} * \chi_{\lambda\overline{B}_s} \|_{\infty} &\le \|\chi_{(1-\lambda)A_r} - \chi_{(1-\lambda)\overline{A}_r}\|_1 + \|\chi_{\lambda B_s} - \chi_{\lambda \overline{B}_s}\|_1 \cr
&\lesssim \eps^{\frac{\alpha_{\tau}}{2^{15}}} \qquad \forall \, r,s \in  \overline{S}. \cr 
\end{split}\end{equation}
On the other hand, by the definition of $\overline{S}$ and the fact that we are analyzing $\overline{C}_t^1,$ we have that  $$\min\{(1-\lambda)\H^1(\overline{A}_r),\lambda \H^1(\overline{B}_s)\} \ge \tau \eps^{\gamma_0}.$$ 
We thus have the convolution estimate 
\begin{equation}\label{eq:convolve-estimate}
\chi_{(1-\lambda)\overline{A}_r} * \chi_{\lambda \overline{B}_s}(x) > 3 \eps^{2\gamma_0}
\end{equation}
whenever 
\[
x \in \left( (1-\lambda)a_f(r) + \lambda a_g(s)+ 3 \eps^{2\gamma_0} , (1-\lambda)b_f(r) + \lambda b_g(s) - 3 \eps^{2\gamma_0}\right). 
\]
Since $(1-\lambda)a_f(r) + \lambda a_g(s) \le - \eps^{\gamma_0},$ $(1-\lambda)b_f(r) + \lambda b_g(s) \ge  \eps^{\gamma_0}$, and $r,s \in (\eps^{\gamma_0},r_0)$, due to the fact that the barycenters of 
$\overline{A}_{r_0}$ and $\overline{B}_{r_0}$ coincide with  the origin, we have that  the set
\[
\left( (1-\lambda)a_f(r) + \lambda a_g(s)+ 3 \eps^{2\gamma_0} , (1-\lambda)b_f(r) + \lambda b_g(s) - 3 \eps^{2\gamma_0}\right) 
\]
contains $(1-\eps^{\frac{\gamma_0}4}) \left( (1-\lambda)\overline{A}_r + \lambda \overline{B}_s\right)$
whenever $\gamma_0 < \alpha_{\tau}/2^{15}.$ 

On the other hand, \eqref{eq:convolution-bound-2} and \eqref{eq:convolve-estimate} imply that 
$$x \in \supp(\chi_{(1-\lambda)A_r} * \chi_{\lambda B_s}) = (1-\lambda)A_r + \lambda B_s.$$ 
Thus, 
\[
(1-\lambda)\overline{A}_r + \lambda\overline{B}_s \subset \frac{1}{1-\eps^{\frac{\gamma_0}4}} ((1-\lambda)A_r+ \lambda B_s )\subset \frac{1}{1-\eps^{\frac{\gamma_0}4}} \{ h > t \},
\]
hence
\[
\overline{C}_t^1 \subset \frac{1}{1-\eps^{\frac{\gamma_0}4}} C_t. 
\]

\noindent\textit{Case 2: Analysis of $\overline{C}_t^2 \cup \overline{C}_t^3.$} Recall that, by assumption, $t > \eps^{\frac{\tau\gamma_0}2}.$ Hence, since $\|\overline{f}\|_{\infty}, \|\overline{g}\|_{\infty} \le 2,$ we readily obtain 
$$r, s \gtrsim \eps^{\frac{\gamma_0}2}.$$
Since $\H^1(\R_+ \setminus \overline{S}) \le \eps^{\gamma_0},$ there exist $r',s' \in \overline{S}$, with $r',s' \in( \eps^{\gamma_0},r_0),$ such that 
$|r-r'| + |s-s'| \le \eps^{\gamma_0}$ and $r > r', \, s > s'.$ Therefore, 
\[
(1-\lambda)\overline{A}_r + \lambda\overline{B}_s \subset (1-\lambda)\overline{A}_{r'} + \lambda\overline{B}_{s'}  \subset \frac{1}{1-\eps^{\frac{\gamma_0}4}} \{ h >(r')^{1-\lambda} (s')^{\lambda} \} \subset \frac{1}{1-\eps^{\frac{\gamma_0}4}} \{ h > t - \eps^{\tau \gamma_0}\},
\]
which implies 
\[
\overline{C}_t \subseteq \frac{1}{1- \eps^{\frac{\gamma_0}4}} \{ h > t - \eps^{\tau \gamma_0}\}, \qquad \forall \, t > \eps^{\frac{\tau\gamma_0}2}. 
\]
Moreover, since $\supp (\overline{h}) \subset (1-\lambda)\supp(\overline{f}) + \lambda \supp(\overline{g})$ and all sets involved are intervals, $\H^1(\supp(\overline{h})) \lesssim \tau^{-4} |\log\eps|^{\frac4\tau}.$ Thus, 
\begin{multline*}
\int_{\R} \overline{h} = \int_0^{\infty} \H^1(\{\overline{h} > t\}) \, d t \le \int_0^{\frac12 \eps^{\frac{\tau\gamma_0}2}}\H^1(\supp(\overline{h})) \, d t \\+ \frac{1}{1-\eps^{\frac{\gamma_0}4}} \int_{\frac12 \eps^{\frac{\tau\gamma_0}2}}^{\infty} \H^1(\{ h > t\})\, d t \le 1+ \frac{c}{\tau^4} \eps^{\frac{\tau\gamma_0}2} |\log\eps|^{\frac4\tau},
\end{multline*}
for some absolute constant $c>0$. This concludes Step 1, as long as we take  $\gamma \in (0,\frac{\tau\gamma_0}2)$ and $c(\tau) = \tau^{-4}.$ \\

\noindent\textbf{$\bullet$ Step 2: the functions $a_f,a_g,b_f,b_g$ are suitably close to satisfying $4-$point inequalities.} We now use similar methods to the ones employed in Section~\ref{sec:rearranged-version} in order to conclude that the functions we constructed are close to being concave. 

Indeed, for notational simplicity, we reset our construction from the beginning, additionally assuming the reductions and conclusions of  Step 1 to hold. In other words, we assume that 
$f,g,h$ satisfy \eqref{eq:condition} and \eqref{eq:almost-eq}, and moreover the level sets of $f,g$ are intervals. We further assume that $\|f\|_{\infty} = 1, \int_{\R} f = \int_{\R} g = 1,$ as in Section~\ref{sec:tail-estimates}. 

Once more, we employ the arguments in Sections~\ref{sec:tail-estimates} and~\ref{sec:rearranged-version}, together with Chebyshev's inequality, in order to show that there is a set of times $F \subset (0,+\infty)$ such that $\H^1(\R_+ \setminus F) \lesssim \eps^{\frac14},$ and moreover 
\[
\left|\H^1(C_t) -(1-\lambda) \H^1(A_t) -\lambda  \H^1(B_t)\right| \lesssim \tau^{-\frac32} \eps^{\frac14}, \qquad \forall \, t \in F. 
\]
Repeating the argument at the beginning of Step 1, we obtain the existence of a set $F' \subset (0,+\infty)$ such that 
\begin{enumerate}
\item 
$\H^1(\R_+ \setminus F') \lesssim \eps^{\delta}, \text{ whenever } \delta < \alpha_{\tau}/1024;$
 \item 
$
\left|\H^1(C_t) -(1-\lambda) \H^1(A_t) -\lambda  \H^1(B_t)\right| \lesssim \tau^{-\frac32} \eps^{\frac14}$  for all $t \in F'$; 
 \item 
 $
 \min\{\H^1(A_t),\H^1(B_t)\} \ge \eps^{\delta}$ for all $t \in (0,1+c\tau^{-\frac32} \eps^{\frac12}) \cap F'$, $\delta \le \alpha_{\tau}/1024$. 
\end{enumerate}
Hence, we define the set $\mathcal{F}_M' ":= \log(F') \cap[-M,M],$ $M = \theta\ \log(1/\eps)$ ($\theta < \delta/2$ to be chosen later). We see, from this definition and a change of variables, $\H^1([-M,M] \setminus \mathcal{F}_M') \lesssim \eps^{\frac{\delta}2},$ and  $\mathcal{F}_M'$ is such that the sets 
\begin{equation*}
\mathcal{A}_R = A_{e^R}=(\mathbf{a}_f(R),\mathbf{b}_f(R)),\quad
\mathcal{B}_S = B_{e^S}=(\mathbf{a}_g(S),\mathbf{b}_g(S)),\quad
\mathcal{C}_T = B_{e^T}=(\mathbf{a}_h(T),\mathbf{b}_h(T)), 
\end{equation*}
 satisfy
  \begin{equation}\label{eq:level-sets-general-almost}
  \left|\H^1(\mathcal{C}_T) - (1-\lambda)\H^1(\mathcal{A}_T) -\lambda  \H^1(\mathcal{B}_T )\right| \lesssim \tau^{-\frac32} \eps^{\frac14}, \qquad \forall \, T \in \mathcal{F}_M'
  \end{equation}
and 
 \begin{equation}\label{eq:level-sets-bound-below-general}
  \min\{\H^1(\mathcal{A}_T),\H^1(\mathcal{B}_T)\} \ge \eps^{\delta}, \qquad \forall \,T\,  \in (-\infty,\log(1+c\tau^{-\frac32} \eps^{\frac12})) \cap \mathcal{F}_M'.
 \end{equation}
We claim that, for $R,S,T \in \mathcal{F}_M'$ are so that $\mathcal{A}_R,\mathcal{B}_S \neq \emptyset,\, (1-\lambda)R+\lambda S=T,$ then  \begin{equation}\label{eq:almost-contained}
(1-\lambda)\mathcal{A}_R + \lambda \mathcal{B}_S \subset \left((1-\lambda)\a_f(T) + \lambda \a_g(T) - \frac{1}{1000} \eps^{\delta},(1-\lambda)\a_f(T) + \lambda \a_g(T)+ \frac{1}{1000}\eps^{\delta}\right). 
\end{equation}

Indeed, if this is not the case, then, by \eqref{eq:level-sets-bound-below-general} and the Brunn--Minkowski inequality, 
$$\H^1\left((1-\lambda)\mathcal{A}_R + \lambda\mathcal{B}_S\right) \ge \eps^{\delta},$$
and thus, as all sets involved are intervals, 
\[
\H^1\left(\left((1-\lambda)\mathcal{A}_R + \lambda\mathcal{B}_S\right) \setminus \left((1-\lambda)\mathcal{A}_T + \lambda \mathcal{B}_T \right) \right) \ge \frac{1}{1000}\eps^{\delta}.
\]
This implies, on the other hand, that
\[
\H^1\left( \mathcal{C}_T \setminus  \left( (1-\lambda) \mathcal{A}_T + \lambda \mathcal{B}_T\right)  \right) \ge \frac{1}{1000} \eps^{\delta},
\]
which, together with \eqref{eq:level-sets-general-almost} and the one-dimensional Brunn--Minkowski inequality, contradicts the definition of $\mathcal{F}_M',$ as long as we take $\eps \ll\tau^3$. Thus, whenever $R,S,T \in \mathcal{F}_M', 
(1-\lambda)R + \lambda S= T,  \mathcal{A}_R, \mathcal{B}_S \neq \emptyset,$ we have 
\begin{equation}\begin{split}\label{eq:three-point-double}
(1-\lambda)\a_f(R) + \lambda \a_g(S) &\ge (1-\lambda)\a_f(T) +\lambda \a_g(T) -\frac{1}{1000}\eps^{\delta},\cr
(1-\lambda) \b_f(R) + \lambda \b_g(S) &\le (1-\lambda) \b_f(T) + \lambda \b_g(T) + \frac{1}{1000} \eps^{\delta},\cr
\end{split}\end{equation}
which proves \eqref{eq:almost-contained}.

As indicated in Section~\ref{sec:rearranged-version}, we can apply \cite[Remark~4.1]{FigalliJerison} to translate the three-point inequalities presented in \eqref{eq:three-point-double} into the following \emph{four-point inequalities}:
\begin{equation}\begin{split}\label{eq:four-point-double}
\a_f(T_1) + \a_f(T_2) &\ge \a_f(T_{1,2}) + \a_f(T_{2,1}) - \frac{1}{\lambda} \eps^{\delta}, \cr 
\a_g(T_1) + \a_g(T_2) &\ge \a_g(T_{1,2}) + \a_g(T_{2,1}) - \frac{1}{\lambda} \eps^{\delta}, \cr 
\end{split}\end{equation}
\begin{equation}\begin{split}\label{eq:four-point-double-2}
\b_f(T_1) + \b_f(T_2) &\le \b_f(T_{1,2}) + \b_f(T_{2,1}) + \frac{1}{\lambda}\eps^{\delta}, \cr
\b_g(T_1) + \b_g(T_2) &\le \b_g(T_{1,2}) + \b_g(T_{2,1}) + \frac{1}{\lambda} \eps^{\delta}, \cr
\end{split}\end{equation}
whenever 
$$T_1,T_2 \in \mathcal{F}_M',\quad T_{1,2} = \frac{1}{2-\lambda} T_1  + \frac{1-\lambda}{2-\lambda} T_2 \in \mathcal{F}_M', \quad T_{2,1} = \frac{1}{2-\lambda} T_2 + \frac{1-\lambda}{2-\lambda} T_1 \in \mathcal{F}_M'.$$
This shows concludes this step, as the functions $a_f,a_g,b_f,b_g$ are close to $\a_f,\a_g,\b_f,\b_g,$ which themselves satisfy four-point inequalities. 

\noindent\textbf{$\bullet$ Step 3: Constructing the log-concave approximations.} We now employ Lemma~\ref{thm:near-concave-lemma} to the functions $\a_f,\a_g,\b_f,\b_g$. 

Indeed, fixing a level $r_0 > 1 - c \eps^{\delta}$ with $\min\{\H^1(\{f>r_0\}),\H^1(\{g>r_0\})\} \ge \eps^{\delta},$ we may suppose that the barycenters of the intervals $\{f > r_0\},\{g > r_0\}$ coincide with the origin; the existence of such a level follows once again by the definition and properties of the set $\mathcal{F}_M'.$ 

 After this reduction, the definition of $\mathcal{F}_M'$ and Lemma~\ref{thm:cutting-support} ensure that the additional hypothesis 
\[
|\a_f(T)| +|\b_f(T)| + |\a_g(T)| + |\b_g(T)| \lesssim \tau^{-4} |\log\eps|^{\frac4\tau}
\]
hold on a subset $\mathfrak{F} \subset \mathcal{F}_M'$ so that $\H^1(\mathcal{F}_M' \setminus \mathfrak{F}) \lesssim \eps^{\delta}$. We thus replace $\mathcal{F}_M'$ by $\mathfrak{F}$, and henceforth still denote it by $\mathcal{F}_M'.$ Notice also that, in such a set, one has $\a_f,\a_g$ nonpositive and $\b_f,\b_g$ nonnegative.

At the present point, one notices that all other prerequisites for Lemma~\ref{thm:near-concave-lemma} are satisfied, thus we may apply it to $\b_f,\b_g,$ and to $-\a_f,-\a_g$ (thanks to \eqref{eq:four-point-double} and \eqref{eq:four-point-double-2}). 

Applying Lemma~\ref{thm:near-concave-lemma} and
arguing as in Section~\ref{sec:rearranged-version}, we find functions $\mathfrak{b}_f,\mathfrak{b}_g,\mathfrak{a}_f,\mathfrak{a}_g$, defined on an interval 
$\Omega_M$ satisfying $\H^1((-M,M) \setminus \Omega_M) \lesssim \eps^{\frac{\delta}2},$ such that 
\begin{equation}\begin{split}\label{eq:first-concave-bound-general}
\int_{\mathcal{F}_M'} |\mathfrak{b}_f(T) - \b_f(T)| \, d T + \int_{\mathcal{F}_M'}|\mathfrak{a}_f(T) - \a_f(T)| \, d T \lesssim \frac{|\log\eps|^{\frac4\tau}}{\tau^{\omega_1}} \eps^{\frac{\delta\alpha_\tau}2},\cr
\int_{\mathcal{F}_M'} |\mathfrak{b}_g(T) - \b_g(T)| \, d T + \int_{\mathcal{F}_M'}|\mathfrak{a}_g(T) - \a_g(T)| \, d T \lesssim  \frac{|\log\eps|^{\frac4\tau}}{\tau^{\omega_1}} \eps^{\frac{\delta\alpha_\tau}2}.\cr
\end{split}\end{equation}
 Moreover, $\mathfrak{b}_f,\mathfrak{b}_g$ are b\emph{concave}, $\mathfrak{a}_f,\mathfrak{a}_g$ are \emph{convex}, and they are all 
bounded in absolute value by $c \tau^{-4} |\log\eps|^{\frac4\tau}$. 

Again, the considerations in Section~\ref{sec:rearranged-version} applied almost verbatim to $\b_f,\b_g, -\a_f,-\a_g$ imply that, by potentially decreasing the power of $\eps$ in the left-hand side of \eqref{eq:first-concave-bound-general}, we may suppose that 
$\mathfrak{a}_f,\mathfrak{a}_g,\mathfrak{b}_f,\mathfrak{b}_g$ are all \emph{monotone} on a smaller interval $I_M = (-3M/4,3M/4),$ and thus, as $\a_f,\a_g,\b_f,\b_g$ are themselves 
bounded by $c \tau^{-4} |\log\eps|^{\frac4\tau},$ 
\begin{equation}\begin{split}
 \int_{I_M}|\mathfrak{a}_f(T) - \a_f(T)| \, d T
 +\int_{I_M} |\mathfrak{b}_f(T) - \b_f(T)| \, d T  \lesssim \frac{|\log\eps|^{1+\frac4\tau}}{\tau^{\frac{3\omega_1}2}} \eps^{\frac{\delta\alpha_\tau}{16}},\cr
 \int_{I_M}|\mathfrak{a}_g(T) - \a_g(T)| \, d T
 +\int_{I_M} |\mathfrak{b}_g(T) - \b_g(T)| \, d T  \lesssim \frac{|\log\eps|^{1+\frac4\tau}}{\tau^{\frac{3\omega_1}2}} \eps^{\frac{\delta\alpha_\tau}{16}}.\cr
\end{split}\end{equation}
Similarly as before, we pick the unique pair $\tilde{f},\tilde{g}$ of functions such that 
\begin{equation*}\begin{split}
\{x \in \R \colon \tilde{f}(x) > t\} = (\mathfrak{a}_f(\log t), \mathfrak{b}_f(\log t)),\qquad
\{x \in \R \colon \tilde{g}(x) > t\} = (\mathfrak{a}_g(\log t), \mathfrak{b}_g(\log t)),
\end{split}\end{equation*}
whenever $\log t \in I_M$ (that is, $t \in (\eps^{\frac{3\theta}4},\eps^{-\frac{3\theta}4})$), 
\begin{equation*}\begin{split}
\text{supp}(\tilde{f}) = \bigcup_{t \in (\eps^{\frac{3\theta}4},\eps^{-\frac{3\theta}4})} (\mathfrak{a}_f(\log t), \mathfrak{b}_f(\log t)), \qquad
\text{supp}(\tilde{g}) = \bigcup_{t \in (\eps^{\frac{3\theta}4},\eps^{-\frac{3\theta}4})} (\mathfrak{a}_g(\log t), \mathfrak{b}_g(\log t)), 
\end{split}\end{equation*}
and  $\{x \in \R \colon \tilde{f}(x) > t\} = \{x \in \R \colon \tilde{g}(x) > s\}  = \emptyset$ for $t,s > \eps^{-\frac{3\theta}4}$ or whenever $\mathfrak{a}_f(\log t) = \mathfrak{b}_f(\log t) = 0 = \mathfrak{a}_g (\log s) = \mathfrak{b}_g(\log s).$

It follows from the convexity of $\mathfrak{a}_f,\mathfrak{a}_g,$ concavity of $\mathfrak{b}_f,\mathfrak{b}_g$ and the argument in Section~\ref{sec:rearranged-version} that these functions are log-concave. \\

\noindent\textbf{$\bullet$ Step 4: Conclusion.} We can finally conclude the proof.  Assume, as in previous sections, that $\|f\|_1 = \|g\|_1 = 1$ and $\min\{\|f\|_{\infty},\|g\|_{\infty}\} = \|f\|_{\infty} = 1.$ Moreover, we assume that \textbf{Steps 1, 2, 3} hold. Thus, using the functions $\tilde{f},\tilde{g}$ and the way we built them, we are led to estimate: 
\begin{equation}\begin{split}\label{eq:layer-cake-3} 
\|f - \tilde{f}\|_1 & = \int_0^{\infty} \H^1(\{ f > t\} \Delta \{ \tilde{f} > t\}) \, d t \cr 
   				    & \le \int_{I_M} |\a_f(T) - \mathfrak{a}_f(T)| \, e^T \, d T + \int_{I_M} |\b_f(T) - \mathfrak{b}_f(T)| e^T \, d T 
     				     + \int_0^{\eps^{\theta}} \H^1(\{ f > t \}) d t\\
     				     & \le \eps^{-\frac{3\theta}4} \left(\int_{I_M} |\a_f(T) - \mathfrak{a}_f(T)| \, d T +\int_{I_M} |\b_f(T) - \mathfrak{b}_f(T)| \, d T\right)  + 	\frac{c}{\tau^4} \eps^{\theta} |\log\eps|^{\frac4\tau} \\
     				     & \lesssim |\log\eps|^{1+\frac4\tau}\tau^{-\frac{3\omega_1}2}\eps^{\frac{\delta\alpha_\tau}{32}} \lesssim  \tau^{-\frac{3\omega_1}2} \eps^{\frac{\delta\alpha_\tau}{64}} , \cr 
\end{split}\end{equation}
by choosing $\theta = \frac{4}{3} \frac{\delta\alpha_\tau}{32}$ and using $\eps \ll e^{-10^{10} \frac{|\log \tau|^4}{\tau^4}}.$
Note that, in this computation, we  assumed $f$ and $g$ to fulfill the requirements in Steps 1-3. In doing so, we lose powers of $\eps$ along the way. More precisely, combining estimates from Section~\ref{sec:rearranged-version} and Steps 1-3, we have: 
\begin{enumerate}
\item We must not incorporate any further power from Section~\ref{sec:rearranged-version}, as it has only been used in the reduction to the case of functions whose level sets are intervals; 
\item In Steps 1-3, we must substitute $\eps \mapsto \frac{c}{\tau^4} \eps^{\frac{\tau \alpha_{\tau}}{2048}},$ by the reduction made in Step 1. 
\end{enumerate}
Thus, we conclude that if the functions $f,g,h$ satisfy \eqref{eq:condition} and \eqref{eq:almost-eq}, then there are log-concave functions $\tilde{f},\tilde{g}$ such that 
$$ \|f-\tilde{f}\|_1 + \|g-\tilde{g}\|_1 \le c\tau^{-\frac{3\omega_1}2} \eps^{\frac{\tau \alpha_{\tau}^3}{2^{30}}} =:  c\tau^{-\frac{3\omega_1}2} \eps^{Q_0(\tau)}.  $$
We are now in a position to use Proposition~\ref{thm:close-stable}. We choose $\eta = c\tau^{\frac{-3\omega_1}{2}} \eps^{Q_0(\tau)}.$ The condition $\eta < c' \tau^3$ for some $c' \in (0,1)$ becomes 
\begin{equation}\label{eq:smallness-epsilon} \eps \le c e^{-M(\tau)},
\end{equation}
where we define $M(\tau) = 10^{40} \omega_1 \frac{|\log(\tau)|^4}{\tau^4},$ and $c>0$ is an absolute constant. Under that condition, notice that all the smallness conditions in the proof above are also fulfilled.

Hence, thanks to Proposition~\ref{thm:close-stable} and the smallness condition \eqref{eq:smallness-epsilon}, there exists a log-concave function $\tilde{h}$ such that, for $f,g,h$ satisfying \eqref{eq:condition} and \eqref{eq:almost-eq}, if we let $a = \|g\|_1/\|f\|_1,$ then there is $w \in \R$ for which 
\begin{eqnarray*}
&\int_{\R}  |a^\lambda f(x) - \tilde{h}(x-\lambda\,w)| \, d x \lesssim \tau^{-\omega_2} \eps^{\frac{Q_0(\tau)}{32}}  \int_\R h,\\
&\int_{\R}  |a^{\lambda-1}g(x) -\tilde{h}(x+(1-\lambda)w)|  \lesssim \tau^{-\omega_2} \eps^{\frac{Q_0(\tau)}{32}}   \int_\R h,\\
&\int_{\R}  |h(x) - \tilde{h}(x)| \, d x \lesssim \tau^{-\omega_2} \eps^{\frac{Q_0(\tau)}{8}} 
\int_\R h.
\end{eqnarray*}
Here, we have let $\omega_2 = \frac{\omega_1}{8} + 2.$ Thus, noting the choices of $Q(\tau),M(\tau)$ in the statement of Theorem~\ref{thm:dim-1-stab}, we notice that this finishes the proof of that result, and thus also the proof of Theorem~\ref{PLstab} in dimension $n=1.$ 

\section{The high-dimensional case}\label{sec:hi-d-case} With the one-dimensional case already resolved in the previous section, we now employ a recent strategy by the first author and A. De \cite{BoroczkyDe} in order to reduce the higher-dimensional version to the one-dimensional one, with the aid of the stability version of the Brunn--Minkowski inequality proved by the second author and D. Jerison \cite{FigalliJerison}. Indeed, we note that the main result in one-dimension implies the following result:

\begin{corollary}\label{thm:multiplicative} Let $F,G,H:\R_+ \to \R_+$ be measurable functions such that 
\begin{equation}\label{eq:condition-2}
H(r^{1-\lambda} s^{\lambda}) \ge F(r)^{1-\lambda} G(s)^{\lambda}, \qquad \forall \, r,s \ge 0,
\end{equation}
where $\lambda \in [\tau,1-\tau]$ for some $\tau \in (0,1/2].$ Suppose that 
\begin{equation}\label{eq:almost-eq-2}
\int_{\R_+} H \le (1+\eps) \left(\int_{\R_+} F\right)^{1-\lambda} \left( \int_{\R_+} G \right)^{\lambda}
\end{equation}
holds for $0<\eps < e^{-M(\tau)}$. Then there are constant $a,b >0$, with $a/b = \|F\|_1/\|G\|_1$, such that 
\[
\int_{\R_+} |a^{-\lambda} F(b^{-\lambda} t) - H(t)| \, d t + \int_{\R_+} |a^{(1-\lambda)} G(b^{(1-\lambda)}t) - H(t)| \, d t
\lesssim \tau^{-\omega} \eps^{Q(\tau)}  \int_{\R_+} H. 
\]
Here, $\omega$ and $Q(\tau)$ are the same as in Theorem~\ref{thm:dim-1-stab}. 
\end{corollary}

\begin{proof} We change variables and define $f(x) = F(e^x) e^x, \, g(x) = G(e^x) e^x, \, h(x) = H(e^x) e^x.$ These functions satisfy \eqref{eq:condition}, and, as 
\[
\int_{\R} f = \int_{\R_+} F, \, \int_{\R} g = \int_{\R_+} G, \, \int_{\R} h = \int_{\R_+} H,
\]
they also satisfy \eqref{eq:almost-eq}. By the result in Section~\ref{sec:general-case}, there is a constant $\eta \in \R$ such that 
\begin{equation*}\begin{split}
\int_{\R} |f(x) - (\|f\|_1/\|g\|_1)^{\lambda} h(x + \lambda \eta)| \, d x \lesssim \tau^{-\omega} \eps^{Q(\tau)} \|f\|_1, \cr
\int_{\R} |g(x) - (\|g\|_1\|f\|_1)^{1-\lambda} h(x+(\lambda-1)\eta)| \, d x \lesssim \tau^{-\omega} \eps^{Q(\tau)} \|g\|_1,
\end{split}\end{equation*}
for $Q(\tau)$ as in the statement of Theorem \ref{thm:dim-1-stab}. Changing variables back, we obtain 
\begin{equation*}\begin{split}
\int_{\R} |F(t) - e^{\lambda \eta}(\|F\|_1/\|G\|_1)^{\lambda} H(te^{\lambda \eta})| \, d t \lesssim \tau^{-\omega} \eps^{Q(\tau)} \|F\|_1, \cr
\int_{\R} |G(t) - e^{(\lambda - 1)\eta}(\|G\|_1\|F\|_1)^{1-\lambda} H(te^{(\lambda - 1)\eta})| \, d t \lesssim \tau^{-\omega} \eps^{Q(\tau)}  \|G\|_1,
\end{split}\end{equation*}
which implies that 
\begin{equation*}\begin{split}
\int_{\R} |e^{-\lambda \eta}(\|G\|_1/\|F\|_1)^{\lambda} F(e^{-\lambda\eta}s) - H(s)| \, d t \lesssim \tau^{-\omega} \eps^{Q(\tau)}  \|F\|_1^{1-\lambda} \|G\|_1^{\lambda}, \cr
\int_{\R} |e^{(1-\lambda)\eta}(\|F\|_1/\|G\|_1)^{1-\lambda}G(e^{(1-\lambda)\eta}s) - H(s)| \, d t \lesssim \tau^{-\omega} \eps^{Q(\tau)}  \|F\|_1^{1-\lambda} \|G\|_1^{\lambda}. \cr
\end{split}\end{equation*}
Taking $a =  \frac{e^{\eta} \|F\|_1}{\|G\|_1}, \, b = e^{\eta}$ and using the Pr\' ekopa--Leindler inequality on the right-hand side of the last expression implies the result. 
\end{proof}
 
Let $f,g,h:\R^n \to \R_+$ satisfy the $n-$dimensional version of \eqref{eq:condition}. We use Corollary~\ref{thm:multiplicative} for the triple $F,G,H$ defined by 
\begin{equation*}\begin{split}
\H^n(\{ x \in \R^n \colon f(x) > t\}) &= F(t), \cr 
\H^n(\{ x \in \R^n \colon g(x) > t\}) &= G(t), \cr 
\H^n(\{ x \in \R^n \colon h(x) > t\}) &= H(t). \cr
\end{split}\end{equation*}
By \eqref{eq:condition} and the $n-$dimensional Brunn-Minkowski inequality, we have 
$$H(r^{1-\lambda} s^{\lambda}) \ge \left((1-\lambda)F(r)^{1/n} + \lambda G(s)^{1/n} \right)^n,$$ 
whenever $F(s),G(r) > 0.$ Thus, using the weighted inequality between arithmetic and geometric means, we get the condition \eqref{eq:condition-2} for $F(s),G(r) > 0.$ Whenever
one of them is zero, \eqref{eq:condition-2} holds trivially, and thus we have verified \eqref{eq:condition-2}. By layer-cake representation, \eqref{eq:almost-eq-2} follows at 
once from \eqref{eq:almost-eq}. 

As conditions are verified, we are in position to use the following result:

\begin{lemma}\label{thm:level-set-measure-close}
If  $\eps \in (0,e^{-M_n(\tau)})$, and $f,g,h:\R^n\to\R_+$ satisfy \eqref{eq:condition}, \eqref{eq:almost-eq} and $\int_{\R^n}f=\int_{\R^n}g=1$,
then there is a dimensional constant $c_n > 0$ such that 
\begin{equation}
\label{eq:levelset-difference}
\int_0^\infty\left|F(t)-H(t)\,\right|\,d t + \int_0^\infty\left|G(t)-H(t)\,\right|\,d t \leq  c_n \tau^{-\frac{\omega}2-1} \eps^{\frac{Q(\tau)}2} .
\end{equation}
\end{lemma}

\begin{proof} In what follows, we let, in analogy to the notation employed in sections~\ref{sec:tail-estimates},~\ref{sec:rearranged-version} and~\ref{sec:general-case}, 
\begin{equation*}\begin{split}
\{ x \in \R^n \colon f(x) > t\} &= A_t, \cr 
\{ x \in \R^n \colon g(x) > t\} &= B_t, \cr 
\{ x \in \R^n \colon h(x) > t\} &= C_t \cr
\end{split}\end{equation*}
denote the level sets of $f,g,h$, respectively. Since $\|f\|_1= \|g\|_1 = 1,$  $\int_0^\infty H=\int_{\R^n}  h\leq 1+\varepsilon$,
it follows from Corollary~\ref{thm:multiplicative} that
there exists some $b>0$ such that
\begin{equation}
\label{FHGH}
\int_0^\infty|b^{\lambda} F(b^{\lambda}t)-H(t)|\,dt+\int_0^\infty|b^{-(1-\lambda)}G(b^{-(1-\lambda)}t)-H(t)|\,dt \leq  a(\tau,\eps),
\end{equation}
where we denote $a(\tau,\eps) = c\tau^{-\omega} e^{Q(\tau)}.$ We may assume, without loss of generality, that $b\geq 1$.

For $t>0$, let
\begin{eqnarray*}
\widetilde{A}_t&=&
b^{\frac{\lambda}{n}}A_{b^{\lambda}t}
\mbox{ \ \ if $\widetilde{A}_t\neq\emptyset$} \\
\widetilde{B}_t&=&
b^{\-\frac{-(1-\lambda)}{n}}B_{b^{-(1-\lambda)}t} 
\mbox{ \ \ if $\widetilde{B}_t\neq\emptyset$}.
\end{eqnarray*}
These sets satisfy $|\widetilde{A}_t|=b^{\lambda} F(b^{\lambda} t)$, $|\widetilde{B}_t|=b^{-(1-\lambda)}G(b^{-(1-\lambda)}t)$ and
\begin{equation}
\label{tildePhiH}
\int_0^\infty|\,|\widetilde{A}_t|-H(t)|\,dt+
\int_0^\infty|\,|\widetilde{B}_t|-H(t)|\,dt\leq  a(\tau,\eps).
\end{equation}
In addition, we also know from the Pr\'ekopa--Leindler condition that
\begin{equation}
\label{widetildeOmega}
(1-\lambda) b^{\frac{-\lambda}n}\widetilde{A}_t
+\lambda b^{\frac{1-\lambda}n}\widetilde{B}_t\subset C_t.
\end{equation}

We proceed to divide the positive line $[0,\infty)$ into two sets where the measures of $\tilde{A}_t, \tilde{B}_t$ are either both close to that of $H(t),$ and otherwise. Indeed, we write  $[0,+\infty) = I \cup J$,
where $t\in I$ if $\frac34\,H(t)<|\widetilde{A}_t|<\frac54\,H(t)$
and $\frac34\,H(t)<|\widetilde{B}_t|<\frac54\,H(t)$, and
$t\in J$ otherwise. 
For $J$, since $\eps < e^{-M_n(\tau)}$, 
\eqref{tildePhiH} yields 
\begin{equation}
\label{Jsize}
\int_J H(t)\,d t\leq 4\int_J\left(|\,|\widetilde{A}_t|-H(t)|+
|\,|\widetilde{B}_t|-H(t)|\right)\,dt
\leq 8a(\tau,\eps) <\frac12.
\end{equation}

Turning to $I$, it follows from the Pr\'ekopa-Leindler inequality
and \eqref{Jsize} that
\begin{equation}
\label{Isize}
\int_I H(t)\,dt\geq 1-\int_J H(t)\,dt>\frac12.
\end{equation}
For $t\in I$, we define
 $\alpha(t)=|\widetilde{A}_t|/H(t)$ and
$\beta(t)=|\widetilde{B}_t|/H(t)$, and hence
$\frac34<\alpha(t),\beta(t)<\frac54$, and
\eqref{tildePhiH}  implies
\begin{equation}
\label{alphabeta}
\int_0^\infty H(t)\cdot\left(|\alpha(t)-1|+|\beta(t)-1|\right)\,dt\leq 2a(\tau,\eps).
\end{equation}

We then proceed by estimating, by the Brunn--Minkowski inequality, 
\begin{eqnarray}
\nonumber
H(t)&\geq& \left((1-\lambda) |A_{b^{\lambda} t}|^{\frac1n}+\lambda |B_{b^{\lambda -1}t}|^{\frac1n}\right)^n
=\left((1-\lambda)b^{\frac{-\lambda}n}|\widetilde{A}_{t}|^{\frac1n}+ \lambda
b^{\frac{1-\lambda}n}|\widetilde{B}_{t}|^{\frac1n}\right)^n\\
\nonumber
&=& |\widetilde{A}_{t}|^{1-\lambda} \cdot |\widetilde{B}_{t}|^{\lambda} 
\left( (1-\lambda) b^{-\frac{\lambda}n} \frac{|\widetilde{A}_t|^{\frac{\lambda}n}}{|\widetilde{B}_t|^{\frac{\lambda}n}} + \lambda b^{\frac{1-\lambda}n} \frac{|\widetilde{B}_t|^{\frac{1-\lambda}{n}}}{|\widetilde{A}_t|^{\frac{1-\lambda}{n}}} \right)^n \\
\label{Halphabetaeta}
&= &H(t)\cdot \alpha(t)^{1-\lambda} \cdot \beta(t)^{\lambda} \left( (1-\lambda) \gamma^{\frac{\lambda}{n}} + \lambda \gamma^{-\frac{1-\lambda}{n}} \right)^n,
\end{eqnarray}
where we let $\gamma = \frac{|\widetilde{A}_t|}{b |\widetilde{B}_t|}.$ Then \eqref{Holderstab} yields 
\[
(1-\lambda)\gamma^{\frac{\lambda}n} + \lambda \gamma^{-\frac{1-\lambda}{n}} \ge 1 + \tau \left( \gamma^{\frac{\lambda}{2n}} - \gamma^{-\frac{1-\lambda}{2n}}\right)^2 \ge 1 + \tau \left(\gamma^{\frac{1}{4n}} - \gamma^{-\frac{1}{4n}}\right)^2.
\]
We now note that for $s\geq 1$, we have 
$$
s^{\frac1{4n}}-s^{-\frac{1}{4n}}=s^{-\frac{1}{4n}}(s^{\frac1{2n}}-1)\geq
s^{-\frac{1}{4n}}\cdot \frac{s^{\frac1{2n}-1}}{2n}(s-1)\geq \frac1{2n}\left(s-\frac1{s}\right),
$$
and thus \eqref{Halphabetaeta} implies 
\begin{equation}\label{eq:followup} 
H(t) \geq H(t)\cdot \alpha(t)^{1-\lambda} \cdot \beta(t)^{\lambda} \left( 1+ \frac{\tau}{4n} \left(\gamma - \gamma^{-1}\right)^2 \right). 
\end{equation}
We claim that if $t\in I$, then
\begin{equation}
\label{alphabetaetabclaim}
\alpha(t)^{1-\lambda} \cdot \beta(t)^{\lambda} \left( 1+ \frac{\tau}{4n} \left(\gamma - \gamma^{-1}\right)^2 \right) \geq 
1-2|\alpha(t)-1|-2|\beta(t)-1|+
\tau \frac{(\sqrt{b}-1)^2}{8n\cdot b}.
\end{equation}
Since
$\alpha(t)^{1-\lambda} \cdot \beta(t)^{\lambda} \geq 1-|\alpha(t)-1|-|\beta(t)-1|$,
\eqref{alphabetaetabclaim} readily holds if $|\alpha(t)-1|+|\beta(t)-1|\geq \frac{(\sqrt{b}-1)^2}{16n\cdot b}$.
Therefore we may assume that
\begin{equation}
\label{alphabetaetabclaimcond}
|\alpha(t)-1|+|\beta(t)-1|\leq
\frac{(\sqrt{b}-1)^2}{16n\cdot b}<\frac12,
\end{equation}
which condition in turn yields that
\begin{equation}
\label{balphabetalow}
\frac{b\beta(t)}{\alpha(t)}\geq
\frac{b\left(1-\frac{(\sqrt{b}-1)^2}{16n^2\cdot b}\right)}{1+\frac{(\sqrt{b}-1)^2}{16n\cdot b}}\geq
b\left(1-2\cdot \frac{(\sqrt{b}-1)^2}{32n\cdot b}\right)\geq
b\left(1-\frac{\sqrt{b}-1}{\sqrt{b}}\right)=\sqrt{b}.
\end{equation}
We deduce first applying  \eqref{alphabetaetabclaimcond}, and then \eqref{balphabetalow} and the fact that $\gamma = \frac{\alpha(t)}{b \beta(t)}$, that
\begin{eqnarray*}
\alpha(t)^{1-\lambda} \cdot \beta(t)^{\lambda} \left( 1+ \frac{\tau}{4n} \left(\gamma - \gamma^{-1}\right)^2 \right)&\geq&
(1-|\alpha(t)-1|-|\beta(t)-1|)\left( 1+ \frac{\tau}{4n} \left(\gamma - \gamma^{-1}\right)^2 \right)\\
&\geq & 1-|\alpha(t)-1|-|\beta(t)-1|+\frac{\tau}{8n} \left(\gamma - \gamma^{-1}\right)^2\\
&\geq & 1-|\alpha(t)-1|-|\beta(t)-1|+\frac{\tau}{8n}\left(\sqrt{b}-\frac1{\sqrt{b}} \right)^2,
\end{eqnarray*}
proving \eqref{alphabetaetabclaim} also under the assumption \eqref{alphabetaetabclaimcond}, as well.

It follows first from \eqref{Isize}, after that from 
\eqref{Halphabetaeta} and \eqref{alphabetaetabclaim} and finally from \eqref{alphabeta} that
\begin{eqnarray*}
\frac{(\sqrt{b}-1)^2}{16n\cdot b}&\leq&\int_I H(t)\cdot \frac{(\sqrt{b}-1)^2}{8n\cdot b}\,dt\leq
\frac{1}{\tau} \int_I H(t)\cdot (2|\alpha(t)-1|+2|\beta(t)-1|)\,dt\\
&\leq & \frac{4 a(\tau,\eps)}{\tau}.
\end{eqnarray*}
Since $\eps < e^{-M_n(\tau)}$, we deduce that $b<2$; therefore, one easily deduces that
\begin{equation}
\label{best}
b\leq 1+50n^{\frac12}\tau^{-\frac12}a(\tau,\eps)^{\frac12}. 
\end{equation}

Next we claim that
\begin{equation}
\label{PhitildePhi}
\int_0^\infty\left|\,|A_t|-|\widetilde{A}_t|\,\right|\,dt+ \int_0^\infty\left|\,|B_t|-|\widetilde{B}_t|\,\right|\,dt\leq 200 n^{\frac12}\tau^{-\frac12}a(\tau,\eps)^{\frac12}. 
\end{equation}
Since $|A_{b^{\lambda}t}|\leq |A_t|$, we have
\begin{eqnarray*}
\int_0^\infty\left||A_t|-|\widetilde{A}_t|\right|\,dt&=&\int_0^\infty\left||A_t|-b^{\lambda}|A_{b^{\lambda}t}|\right|\,dt\\
&\leq &\int_0^\infty\left||A_t|-b^{\lambda}|A_{t}|\right|\,dt+b^{\lambda} \int_0^\infty\left||A_t|-|A_{b^{\lambda}t}|\right|\,dt\\
&= & (b^{\lambda}-1)+
b^{\lambda} \int_0^\infty\left(|A_t|-|A_{b^{\lambda}t}|\right)\,dt\\
&=& 2(b^{\lambda}-1)\leq 100 \lambda 2^{\lambda-1} n^{\frac12}\tau^{-\frac12}a(\tau,\eps)^{\frac12} \le 100 n^{\frac12}\tau^{-\frac12}a(\tau,\eps)^{\frac12}. 
\end{eqnarray*}
Similarly, $|B_t|\leq |B_{b^{\lambda-1}t}|$, and hence
\begin{eqnarray*}
\int_0^\infty\left||B_t|-|\widetilde{B}_t|\right|\,dt&=&\int_0^\infty\left||B_t|-b^{\lambda-1}|B_{b^{\lambda-1}t}|\right|\,dt\\
&\leq &\int_0^\infty\left||B_t|-b^{\lambda-1}|B_{t}|\right|\,dt+
b^{\lambda-1} \int_0^\infty\left||B_t|-|B_{b^{\lambda-1}t}|\right|\,dt\\
&= & (1-b^{\lambda-1})+
b^{\lambda-1}\int_0^\infty\left(|B_{b^{\lambda-1}t}|-|B_t|\right)\,dt\\
&=& 2(1-b^{\lambda-1})\leq 100 n^{\frac12}\tau^{-\frac12}a(\tau,\eps)^{\frac12},
\end{eqnarray*}
proving \eqref{PhitildePhi}. We conclude the proof by combining \eqref{tildePhiH} and \eqref{PhitildePhi}.
\end{proof}

As a by-product of Lemma~\ref{thm:level-set-measure-close}, notice that, by setting $\min(\|f\|_{\infty},\|g\|_{\infty}) = \|f\|_{\infty} = 2,$ then 
\[
\tau^{-\frac12}a(\tau,\eps)^{\frac12} \gtrsim_n \int_2^{\max{\|g\|_{\infty},\|h\|_{\infty}}} \left( G(t) + H(t) \right) \, d t.
\]
In particular, we know that 
\begin{equation}\label{eq:level-set-high-d-1}
C_t \supset (1-\lambda)A_t + \lambda B_t
\end{equation}
whenever $t \in (0,2).$ We claim, before proceeding with the proof, that under such conditions, 
\begin{equation}\label{eq:bound-size-g}
\|g\|_{\infty} \le \frac{2e\cdot 3^{n+1}}{\tau^{n+1}}.
\end{equation} Indeed, if $y_0 \in \R^n$ is fixed, we have 
\[
C_t \supset (1-\lambda) A_{t^{\frac{1}{1-\lambda}}/g(y_0)^{\frac{\lambda}{1-\lambda}}} + \lambda y_0.
\]
In particular, 
\begin{equation*}\begin{split}
\int_0^{t} F(s) \, ds & = \frac{1}{1-\lambda} \int_0^{t^{1-\lambda} g(y_0)^{\lambda}} F\left( \frac{r^{1/(1-\lambda)}}{g(y_0)^{\lambda/(1-\lambda)}}\right) \left( \frac{r}{g(y_0)} \right)^{\lambda/(1-\lambda)} \, dr \cr
 & \le \frac{1}{1-\lambda} \left( \frac{t}{g(y_0)} \right)^{\lambda} \int_0^{t^{1-\lambda} g(y_0)^{\lambda}} F\left( \frac{r^{1/(1-\lambda)}}{g(y_0)^{\lambda/(1-\lambda)}}\right) \, dr\cr
 & \le \frac{1}{(1-\lambda)^{n+1}} \left( \frac{t}{g(y_0)} \right)^{\lambda} \int_0^{t^{1-\lambda} g(y_0)^{\lambda}}  H(r) \, dr.
\end{split}\end{equation*}
Therefore, by picking $t = 2$ and using that $\int H \le 1 + \varepsilon, \, \int_0^2 F(s) \, ds = 1,$
\[
g(y_0) \le \frac{2\cdot (1+\eps)^{1/\lambda} }{(1-\lambda)^{(n+1)/\lambda}}.
\] 
A quick analysis shows that, for $\lambda \in (0,1),$ the inequality 
$$
(1-\lambda)^{1/\lambda} \ge \frac{1}{3} (1-\lambda)
$$
holds. If $\eps < \tau,$ then the numerator is at most $2e,$ and thus, as $y_0$ was arbitrary above, we conclude the claim. Using now \eqref{eq:level-set-high-d-1}, we get  
\[
H(t) \ge \left((1-\lambda)F(t)^{1/n} + \lambda G(t)^{1/n}\right)^n \ge \frac{F(t) + G(t)}{2} - \frac{|F(t)-G(t)|}{2} \qquad \forall \, t \in (0,2).
\]
Notice also that, by Lemma~\ref{thm:level-set-measure-close} ,
\[
\int_0^{\infty} |F(t) - G(t)| \, d t \lesssim_n\tau^{-\frac12}a(\tau,\eps)^{\frac12}.
\]
Thus, by these considerations and the almost-optimality of $f,g,h$ for the Pr\'ekopa--Leindler inequality, we obtain 
\begin{equation}\label{eq:upper-bound-high-dim}
c_n\tau^{-\frac12}a(\tau,\eps)^{\frac12} \ge \int_0^{\alpha} \left( H(t) - \frac{F(t)+G(t)}{2} + \frac{|F(t) - G(t)|}{2}\right) \, d t \qquad \forall \, \alpha \ge 0. 
\end{equation}
On the other hand, notice that \eqref{eq:power-containment-levels} implies, together with a limiting argument and the Brunn--Minkowski inequality, 
$$H(t) \ge \max\left\{\left(\lambda G\left(t^{\frac{1}{\lambda}}\right)^{1/n} + (1-\lambda) F(1)^{1/n}\right)^n, \left((1-\lambda)F\left(t^{\frac{1}{1-\lambda}}\right)^{1/n} +\lambda G(1)^{1/n}\right)^n\right\},$$
for all $t \in (0,2)$ so that 
$H(t) > 0.$ Thus, \eqref{eq:upper-bound-high-dim} implies 
\begin{equation}\label{eq:iteration-high-dim}
c_n\tau^{-\frac12}a(\tau,\eps)^{\frac12}  \ge \int_0^{\alpha} \left( \frac{1}{2}\left( (1-\lambda)^n F\left(t^{\frac{1}{1-\lambda}}\right) + \lambda^n G\left(t^{\frac{1}{\lambda}}\right) \right) - \frac{F(t) + G(t)}{2}\right) \, d t. 
\end{equation}
We thus let, in analogy to Lemma~\ref{thm:cutting-support}, 
$$\Gamma(\alpha) = \int_0^{\alpha} ((1-\lambda)^n F(t) + \lambda^n G(t)) \, d t.$$
Again in analogy to Lemma \ref{thm:cutting-support}, we may suppose without loss of generality that $\lambda \le 1/2.$  Then \eqref{eq:iteration-high-dim} implies 
\[
\frac{1-\lambda}{2} \Gamma(\alpha^{\frac{1}{1-\lambda}})\alpha^{-\frac{\lambda}{1-\lambda}} \le c_n\tau^{-\frac12}a(\tau,\eps)^{\frac12}   + \frac{\Gamma(\alpha)}{2\tau^n}.
\]
As in the proof of Lemma~\ref{thm:cutting-support}, we let $\beta = \alpha^{\frac{1}{1-\lambda}}.$ We thus have 
$$
\frac{\Gamma(\beta)}{\beta} \le 2c_n \tau^{-\frac32}a(\tau,\eps)^{\frac12} \cdot \frac{1}{ \beta^{1-\lambda}} + \frac{1}{\tau^{n+1}} \frac{\Gamma(\beta^{1-\lambda})}{\beta^{1-\lambda}},
$$
and therefore 
$$
\frac{\Gamma(\beta)}{\beta} \le \left(2c_n \tau^{-\frac32}a(\tau,\eps)^{\frac12}  \sum_{i=1}^k\frac{(1/\tau^{n+1})^{i-1}}{\beta^{(1-\lambda)^i}}\right)
+(1/\tau^{n+1})^{k}\frac{\Gamma(\beta^{(1-\lambda)^k})}{\beta^{(1-\lambda)^k}}.
$$
We now select $k \in \N$ to be the first natural number such that $\beta^{(1-\lambda)^k} > e^{-1}.$ This implies that
$$
\Gamma(\beta) \lesssim  (1/\tau^{n+1})^k \left( 1 + c_n \frac{\sqrt{a(\tau,\eps)}}{\beta^{1-\lambda} \tau^{\frac32}} \right) \beta .
$$
If $\beta > \eps^{\frac{Q(\tau)}2},$ then the estimate above yields 
$$
\Gamma(\beta) \le c_n \tau^{-\frac{\omega + 3}{2}  } \beta |\log(\beta)|^{\frac{4(n+3) |\log \tau|}{\tau}}.
$$  
In particular, one concludes directly from the definition of $\Gamma$ that 
\begin{equation}\label{eq:control-levels-high-d}
F(\beta) + G(\beta) \le c_n \tau^{-\frac{\omega + 3+n}{2}} |\log\eps|^{\frac{4(n+3)|\log \tau|}{\tau}}, \qquad \forall \, \beta > \eps^{\frac{Q(\tau)}2}.
\end{equation}
We are now ready to give the proof of Theorem~\ref{PLstab} in dimensions $n \ge 2.$ For that, we use the shorthand $\rho_n(\tau) = \frac{4(n+10)|\log \tau|}{\tau}.$ 
\begin{proof}[Proof of Theorem~\ref{PLstab}, $n \ge 2$]
Let $\theta > 0$ be small, to be chosen later. Define the (truncated) log-hypographs of $f,g,h$ as 
\begin{equation*}\begin{split}
\mathcal{S}_f = \{ (x,T) \in \R^{n+1} \colon x \in \{ f > \eps^{\theta}\}, \, \eps^{\theta} \le e^T < f(x)\},\cr
\mathcal{S}_g = \{ (x,T) \in \R^{n+1} \colon x \in \{ g > \eps^{\theta}\}, \, \eps^{\theta} \le e^T < g(x)\},\cr
\mathcal{S}_h = \{ (x,T) \in \R^{n+1} \colon x \in \{ h > \eps^{\theta}\}, \, \eps^{\theta} \le e^T < h(x)\}. 
\end{split}\end{equation*}
We first claim that the measure of the two first of such sets is well-controlled. Indeed, it follows directly from the definition of such sets and \eqref{eq:control-levels-high-d} that, for $\theta < Q(\tau)/4,$
\begin{equation}\label{eq:upper-bound-S_f}
c_n \theta \tau^{-\frac{\omega+3+n}2} |\log\eps|^{\rho_n(\tau)} \ge  \theta |\log\eps| \cdot \H^n(\{f > \eps^{\theta}\}) \ge \H^{n+1}(\mathcal{S}_f).
\end{equation}
On the other hand, by a change of variables and the normalization chosen for $f,$ one obtains
\begin{equation}\label{eq:lower-bound-S_f}
\H^{n+1}(\mathcal{S}_f) = \int_{\theta \log \eps}^{\log \|f\|_{\infty}} F(e^s) \, ds > \frac{1}{2}.
\end{equation}
The same estimates together with \eqref{eq:bound-size-g} show that 
\begin{equation}\label{eq:lower-and-upper-S_g}
c_n \theta \tau^{-\frac{\omega+3+n}2} |\log\eps|^{\rho_n(\tau)} \ge \H^{n+1}(\mathcal{S}_g) > \frac{\tau^{(n+1)}}{2e \cdot 3^{n+1}}.
\end{equation}
holds as well. Employing Lemma~\ref{thm:level-set-measure-close}, we obtain that 

\begin{equation}\begin{split}\label{eq:close-volumes-high-d}
& |\H^{n+1}(\mathcal{S}_f) - \H^{n+1}(\mathcal{S}_h)| + |\H^{n+1}(\mathcal{S}_g) - \H^{n+1}(\mathcal{S}_h)| \cr 
& \le  \int_{\theta \log \eps}^{\infty} \left( |F(e^s) - H(e^s)| +|G(e^s) - H(e^s)| \right) \, ds \cr
& \le  \eps^{-\theta} \left( \int_0^{\infty} \left(|F(t) - H(t)| + |G(t) - H(t)|\right) \, ds \right) \cr 
& \le c_n \tau^{-\frac{\omega+3}{2}} \eps^{\frac{Q(\tau)}2 - \theta} =: \tau^n \cdot \delta(\eps,\tau,\theta).
\end{split}\end{equation}
We denote, until the end of the proof, $\delta = \delta(\eps,\tau,\theta)$ for shortness.  By \eqref{eq:condition}, we have
\begin{equation}\label{eq:Brunn-Minkowski-containment-graph} 
(1-\lambda) \mathcal{S}_f + \lambda \mathcal{S}_g \subset \mathcal{S}_h.
\end{equation}
In particular, \eqref{eq:close-volumes-high-d} and \eqref{eq:Brunn-Minkowski-containment-graph} imply that 
\begin{equation}\label{eq:lower-upper-s-h}
2c_n \tau^{-\frac{\omega+3+n}2} |\log\eps|^{\rho_n(\tau)} \ge \H^{n+1}(\mathcal{S}_h) \ge \frac{\tau^n}{2}.
\end{equation}
We are in position to use Theorem~\ref{BrunnMinkowskistab}. That result states that, under the conditions satisfied by the sets $\mathcal{S}_f, \mathcal{S}_g$ and $\mathcal{S}_h$ in \eqref{eq:upper-bound-S_f}, \eqref{eq:lower-bound-S_f}, \eqref{eq:lower-and-upper-S_g}, \eqref{eq:close-volumes-high-d} and \eqref{eq:Brunn-Minkowski-containment-graph}, then for $\delta < e^{-A_n(\tau)}$, the sets $\mathcal{S}_f,\mathcal{S}_g$ are both close (in quantitative terms of 
$\delta = \delta(\eps,\tau,\theta)$) to their convex hulls. Here, we let $A_n(\tau) = \frac{2^{3^{n+2}} n^{3^n} |\log \tau|^{3^n}}{\tau^{3^n}}$, in accordance to Theorem 1.3 in \cite{FigalliJerison}. 

In more effective terms, Theorem~\ref{BrunnMinkowskistab} implies that there exist an absolute constant $c_n > 0$ and an exponent $\gamma_n(\tau) = \frac{\tau^{3^n}}{2^{3^{n+1}} n^{3^n} |\log \tau|^{3^n}}$ such that the following holds. Denote the closure of the convex hull of $\mathcal{S}_f,\mathcal{S}_g, \mathcal{S}_h$ by $S_f,S_g,S_h$ respectively. There are $\tilde{w} = (w,\omega) \in \R^{n+1},$ and a convex set $\mathbb{S}_h \supset S_h$ with 
\begin{equation}\begin{split}\label{eq:brunn-minkowski-optimal-n+1}
\mathbb{S}_h &\supset (\mathcal{S}_f - \tilde{w}) \cup (\mathcal{S}_g + \tilde{w}),\cr
\H^{n+1}(S_h \setminus \mathcal{S}_h) + \H^{n+1}(S_f \setminus \mathcal{S}_f) & + \H^{n+1}(S_g \setminus \mathcal{S}_g) \le c_n \tau^{-N_n-\frac{\omega+3+n}2} |\log\eps|^{\rho_n(\tau)}\delta^{\gamma_n(\tau)},\cr
\H^{n+1}(\mathbb{S}_h \setminus \mathcal{S}_h) + \H^{n+1}(\mathbb{S}_h \setminus (\mathcal{S}_f - \tilde{w})) & + \H^{n+1}(\mathbb{S}_h \setminus (\mathcal{S}_g + \tilde{w})) \le c_n\tau^{-N_n-\frac{\omega+3+n}2} |\log\eps|^{\rho_n(\tau)} \delta^{\gamma_n(\tau)}.\cr
\end{split}\end{equation}
We thus use the shorthand $N_n' = N_n + \frac{\omega+3+n}{2}.$ Now \eqref{eq:brunn-minkowski-optimal-n+1} readily implies that $\H^{n+1}(\mathbb{S}_h \setminus S_h) \le 2c_n \tau^{-N_n'} |\log\eps|^{\rho_n(\tau)} \delta^{\gamma_n(\tau)},$ and thus 
\begin{equation}\label{eq:close-convex-sets}
\H^{n+1}(S_h \Delta (S_f - \tilde{w})) + \H^{n+1}(S_h \Delta (S_g - \tilde{w})) \le 6 c_n \tau^{-N_n'}|\log\eps|^{\rho_n(\tau)}  \delta^{\gamma_n(\tau)}.
\end{equation}
We now employ the analysis of \cite[Lemma~6.1]{BoroczkyDe}. Explicitly, suppose first $\tilde{w} = (w,\omega), \, \omega > 0.$ We let 
$$S_f^{\omega} = \{(x,T) \in S_f \colon \theta \log \eps \le T \le \theta \log \eps + \omega\}.$$ 
By convexity of $S_f,$ it follows that $\H^{n+1}(S_f \Delta (S_f + (0,\omega))) = 2 \H^{n+1}(S_f^{\omega}).$ But it also follows that 
$S_f^{\omega} \subset S_f \setminus (S_h+\tilde{w}),$ which, by \eqref{eq:brunn-minkowski-optimal-n+1} and \eqref{eq:close-convex-sets}, implies that 
$$\H^{n+1}(S_f^{\omega}) \le 6c_n \tau^{-N_n'} |\log\eps|^{\rho_n(\tau)}  \delta^{\gamma_n(\tau)}.$$
Thus, by triangle inequality, 
\begin{equation*}\begin{split}
\H^{n+1}(S_f \Delta (S_h + (w,0))) \le 2\H^{n+1}(S_f^{\omega}) + \H^{n+1}(S_f \Delta (S_h + \tilde{w})) \le 18 c_n \tau^{-N_n'} |\log\eps|^{\rho_n(\tau)}  \delta^{\gamma_n(\tau)}.
\end{split}\end{equation*}
A similar argument works in case $\omega < 0,$ if one considers $S^{|\omega|}_h$ instead of $S^{\omega}_f.$ In the end, this allows one 
to conclude that there is $w \in \R^n$ so that
\begin{equation}\label{eq:close-convex-sets-dim-n}
\H^{n+1}(S_h \Delta (S_f - w)) + \H^{n+1}(S_h \Delta (S_g + w)) \le 72 c_n \tau^{-N_n'} |\log\eps|^{\rho_n(\tau)}  \delta^{\gamma_n(\tau)}.
\end{equation}
We now note that, as  
$\{f > \eps^{\theta}\} \times \{ T = \theta \log \eps \} \subset \mathcal{S}_f,$ then 
\[
S_f \supset \text{co}(\{f > \eps^{\theta}\}) \times \{ T= \theta \log \eps \}.
\]
We associate to each $x \in \text{co}(\{f > \eps^{\theta}\})$ the function 
$$T_f(x) = \sup\{ T \in \R \colon (x,T) \in S_f\}.$$
This satisfies clearly $T_f(x) \ge \theta \log \eps, \forall x \in \text{co}(\{ f> \eps^{\theta}\}).$ We claim that this function is, moreover, concave. Indeed, if $(x,T_1),(y,T_2) \in S_f,$ by convexity of that set we get 
$$(tx+(1-t)y,tT_1 + (1-t)T_2) \in S_f.$$
Thus,
\begin{equation*}\begin{split} T_f(tx +(1-t)y) &= \sup\{ T \in \R \colon (tx + (1-t)y,T) \in S_f\}\cr
			       &\ge t \sup\{ T \in \R \colon (x,T_1) \in S_f\} + (1-t) \sup\{T \in \R \colon (y,T_2) \in S_f\} \cr 
			       & = t T_f(x) + (1-t)T_f(y), \, \forall \, t \in (0,1).\cr
\end{split}\end{equation*}
By definition of $S_f,$ it also follows that $T_f(x) \ge \log f(x), \, \forall x \in \text{co}(\{f > \eps^{\theta}\}).$ Let 
$$\tilde{f}(x) = \begin{cases}
		  e^{T_f(x)}, & \text{ if } x \in \text{co}(\{ f>\eps^{\theta}\}); \cr
		  0, & \text{ otherwise }.\cr
                 \end{cases}
$$
Now notice that $(x,r)$ belongs to the interior of $S_f$ if and only if $T_f(x) > r > \theta \log \eps$ and $x$ belongs to the interior of $\text{co}(\{ f > \eps^{\theta}\}).$ Writing $A(r) = \{ (x,T) \in A, T = r\}$ for horizontal slices of a set $A \subset \R^{n+1},$ we compute, by Fubini,
\begin{equation}\begin{split}\label{eq:almost-there-high-d-f}
\H^{n+1}(S_f \setminus \mathcal{S}_f) &= \int_{-\infty}^{\infty} \H^n(S_f(r) \setminus \mathcal{S}_f(r)) \, d r \cr 
				      &= \int_{\theta \log \eps}^{\log 2} \H^n(\{\log \tilde{f} > r\} \setminus \{ \log f > r\}) \, d r \cr
				      &= \int_{\eps^{\theta}}^2 \H^n(\{ \tilde{f} > s\} \Delta \{ f > s\}) \, \frac{ d s}{s} \cr
				      &\ge \frac{1}{2} \int_{\eps^{\theta}}^2 \H^n(\{ \tilde{f} > s\} \Delta \{ f > s\}) \, d s. \cr 
\end{split}\end{equation}
By Chebyshev's inequality and \eqref{eq:brunn-minkowski-optimal-n+1}, there is 
$$s_0 \in (\eps^{\theta},   \eps^{\theta} + c_n\tau^{-\frac{N_n'}2} \delta^{\frac{\gamma_n(\tau)}2})$$
$$ \text{ so that } \H^n(\{ \tilde{f} > s_0\} \Delta \{ f > s_0\}) \le  \tau^{-\frac{N_n'}2} |\log\eps|^{\rho_n(\tau)} \delta^{\frac{\gamma_n(\tau)}2}. $$ 
Recalling the definition of $\delta,$ one notices that, if $\frac{Q(\tau)}{4} > \theta,$ and $\eps <(c_n)^{-1} e^{\frac{2^{10} N_n \log (\tau)}{\gamma_n(\tau) Q(\tau)}} $ we may take $s_0 \in (\eps^{\theta}, 2 \eps^{\theta})$ so that 
\begin{equation}\label{eq:difference-levels-f-high-d}
\H^n(\{ \tilde{f} > s_0\} \Delta \{ f > s_0\}) \lesssim \tau^{-N_n^{'}/2} |\log\eps|^{\rho_n(\tau)} \eps^{\frac{\gamma_n(\tau) Q(\tau)}{8}}.
\end{equation}
Define then the function $\tilde{f_1}$ to be zero whenever $\tilde{f} \le s_0,$ and equal to $\tilde{f}$ otherwise. This new function is again log-concave.

We claim that this new function is still sufficiently close to $f.$ Indeed, by gathering \eqref{eq:almost-there-high-d-f}, \eqref{eq:difference-levels-f-high-d} and  \eqref{eq:control-levels-high-d}, we have
\begin{equation}\begin{split}\label{eq:final-high-dim-f}
\| \tilde{f_1} - f\|_1 &= \int_0^2 \H^n(\{ \tilde{f_1} > t\} \Delta \{ f > t\}) \, d t \cr
		     &\le \int_0^{s_0} \left(\H^n(\{ \tilde{f_1} > s_0\}) + \H^n(\{ f > t\})\right) \, d t 
		     + \int_{s_0}^2  \H^n(\{ \tilde{f_1} > t\} \Delta \{ f > t\}) \, d t \cr 
		     &\le c_n \tau^{-\frac{\omega+3+n}2} \eps^{\theta} |\log\eps|^{\rho_n(\tau)} + \int_{s_0}^2 \H^n(\{ \tilde{f} > t\} \Delta \{ f > t\}) \, d t \cr 
		     &\le  c_n \tau^{-\frac{\omega+3+n}2} \eps^{\theta} |\log\eps|^{\rho_n(\tau)} + 2\H^{n+1}(S_f \setminus \mathcal{S}_f) \cr 
		     &\lesssim_n \tau^{-N_n^{'}} \eps^{\frac{\gamma_n(\tau)Q(\tau)}{16}}  |\log\eps|^{\rho_n(\tau)},
\end{split}\end{equation}
where we chose $\theta = \frac{\gamma_n(\tau)Q(\tau)}{16}.$ Fix this value, and thus the value of $\delta,$ for the rest of the proof. Such an inequality is evidently not restrictive to $f,$ and the same argument yields that there is a log-concave function $\tilde{g}_1$ so that 
\begin{equation}\label{eq:final-high-dim-g}
\| \tilde{g}_1 - g\|_1 \lesssim_n \tau^{-N_n^{'} - (n+1)} \eps^{\frac{\gamma_n(\tau)Q(\tau)}{16}}  |\log\eps|^{\rho_n(\tau)}. 
\end{equation}

In order to conclude, we only need to prove that both of $\tilde{f}_1,\tilde{g}_1$ are sufficiently close, after a translation, to a log-concave function $\tilde{h}_1.$  In order to prove that, one only needs to construct the function $\tilde{h}$ in entire analogy to what we did for $\tilde{f}, \tilde{g};$ that is, we let 
\[
T_h(x) = \sup\{T \in \R \colon (x,T) \in S_h\}.
\]
One readily verifies that this new function is, again, concave, and that the function 
\[
\tilde{h}(x) = \begin{cases} 
					e^{T_h(x)}, & \text{ if } x \in \text{co}(\{h > e^{\theta}\}); \cr 
					0, & \text{ otherwise,}
					\end{cases}
\]
is log-concave. Using \eqref{eq:close-convex-sets-dim-n} together with an argument similar to \eqref{eq:final-high-dim-f} implies that 
\begin{multline}\label{eq:comparison-to-same-function}
\H^{n+1}(S_h \Delta (S_f -w)) + \H^{n+1}(S_h \Delta (S_g + w)) \ge \\
 \int_{0}^{\|\tilde{h}_1\|_{\infty}} \left( \H^n(\{ \tilde{h} > s\} \Delta \{ \tilde{f}(\cdot + w) > s\}) +  \H^n(\{ \tilde{h} > s\} \Delta \{ \tilde{g} (\cdot - w)> s\})\right) \, \frac{ds}{s}.  \cr
\end{multline}
Notice now that $\|\tilde{f}_1\|_{\infty} =\|f\|_{\infty}, \|\tilde{g}_1\|_{\infty} =\|g\|_{\infty},$ by construction. The idea is then to truncate from below at height $\{\tilde{h} > s_0\}$ and from above at height $\varrho := \max(\|\tilde{f}_1\|_{\infty}, \|\tilde{g}_1\|_{\infty})$ in order to generate a new function, which is again log-concave by construction. Denote this new function by $\tilde{h}_1.$ Moreover, by \eqref{eq:comparison-to-same-function} in conjunction with \eqref{eq:bound-size-g}, we have 
\begin{equation}\begin{split}\label{eq:log-concave-close-high-d}
 & 2e\cdot 3^{n+1} \tau^{-n-1} c_n \tau^{-N_n^{'}} |\log\eps|^{\rho_n(\tau)}  \delta^{\gamma_n(\tau)}  \cr 
 & \ge \int_{s_0}^{\varrho} \left(\H^n(\{ \tilde{h}_1 > s\} \Delta \{ \tilde{f}_1(\cdot + w) > s\}) +  \H^n(\{ \tilde{h}_1 > s\} \Delta \{ \tilde{g}_1 (\cdot - w)> s\})\right) \, ds \cr 
 & = \int_{\R^n} \left( |\tilde{h}_1(x) - \tilde{f}_1(x+w)| + |\tilde{h}_1(x) - \tilde{g}_1(x-w)| \right) \, dx. \cr
\end{split}\end{equation}
Combining \eqref{eq:final-high-dim-f}, \eqref{eq:final-high-dim-g} and \eqref{eq:log-concave-close-high-d} implies that 
\[
\| \tilde{h}_1 (\cdot - w) - f \|_1 + \| \tilde{h}_1 (\cdot + w) - g\|_1 \lesssim_n \tau^{-N_n' - n - 1 } |\log \eps|^{\rho_n(\tau)} \eps^{\frac{\gamma_n(\tau) Q(\tau)}{16}}.
\]

 Finally, in order to prove that $h$ is close to $\tilde{h}_1,$ we estimate 
\begin{equation}\begin{split}\label{eq:integral-layers-h}
\int_{\R^n}  |h(x) - \tilde{h}_1(x)| \, dx &= \int_0^{s_0} \H^n (\{ h > s\}) \, ds  \cr
			    	& \qquad+ \int_{s_0}^{\varrho} \H^n(\{ h > s\} \Delta \{ \tilde{h}> s\}) \, ds + \int_{\varrho}^{\infty} \H^n (\{ h > s\}) \, ds \cr 
			    	& \le c_n \tau^{-\frac{\omega+3+n}2}  \eps^{Q(\tau)\gamma_n(\tau)/16} |\log\eps|^{\rho_n(\tau)} + \int_{s_0}^{\varrho} \H^n(\{ h > s\} \Delta \{ \tilde{h}_1 > s\}) \, ds \cr 
			    	& + c_n \tau^{-\omega/2} \eps^{\frac{Q(\tau)}2}, 
\end{split}\end{equation}
where we used both \eqref{eq:control-levels-high-d} and Lemma~\ref{thm:level-set-measure-close} in the last line. In order to deal with the middle term, we remark that an anrgument entirely analogous to that of \eqref{eq:almost-there-high-d-f} implies that 
\[
\H^n(S_h \setminus \mathcal{S}_h) \ge \frac{1}{\varrho} \int_{s_0}^{\varrho}\H^n(\{ h > s\} \Delta \{ \tilde{h}> s\}) \, ds,
\]
which on the other hand implies 
\begin{equation}\label{eq:middle-term-h}
\int_{s_0}^{\varrho}\H^n(\{ h > s\} \Delta \{ \tilde{h}_1> s\}) \, ds \lesssim_n \tau^{-n-1} \tau^{-N_n^{'}} \eps^{\gamma_n Q(\tau)/16} |\log\eps|^{\rho_n(\tau)}.
\end{equation}
Inserting \eqref{eq:middle-term-h} into \eqref{eq:integral-layers-h} implies 
\begin{equation}\label{eq:final-bound-h-high-d}
\|h-\tilde{h}_1\|_1 \lesssim_n \tau^{-N_n^{'} - (n+1)} \eps^{\frac{\gamma_n(\tau) Q(\tau)}{16}}|\log\eps|^{\rho_n(\tau)}.
\end{equation}
Finally, in order to arrive at the statement of Theorem \ref{PLstab}, we notice that the expression on the right-hand side of \eqref{eq:final-bound-h-high-d} may be bounded by $c_n \tau^{-N_n' - n-1} \eps^{\frac{\gamma_n(\tau) Q(\tau)}{32}},$ as long as $\eps < e^{-c_n \frac{|\log \tau| \rho_n(\tau)^2}{Q_n(\tau)^2}},$ for $c_n \gg 1$ sufficiently large absolute constant.

An inspection of the constants needed for the proof above allows us conclude that Theorem~\ref{PLstab} holds with $\Sigma_n = N_n + \frac{\omega+3+n}{2} + (n+1),$ as $\tau^{\gamma_n(\tau)}$ is bounded by an explicitly computable absolute constant $\tilde{C}_n$ whenever $\tau \in [0,1].$ We also conclude that we may take $Q_n(\tau) = \frac{Q(\tau) \gamma_n(\tau)}{16},$ and  the result holds whenever $\eps < c_n e^{-M_n(\tau)},$ where $c_n > 0$ is an explicitly computable absolute constant, and one may take 
\[
M_n(\tau) = c_n |\log(\tau)| \max\left\{ \frac{A_n(\tau)}{Q(\tau)}, \frac{ \rho_n(\tau)^2}{Q_n(\tau)^2} \right\},
\]
for $c_n > 0$ a sufficiently large absolute constant, depending only on the dimension $n \ge 2.$ This finishes the proof of the higher-dimensional case, and thus also of Theorem \ref{PLstab}. 
\end{proof}

\end{document}